%% file: _main.tex
\documentclass{article}

\input{_header}

\begin{document}
\maketitle

\input{sec/000Abstract}

\input{sec/001Introduction}
\input{sec/002HardInstance}
\input{sec/003NonAnytime}
\input{sec/004Anytime}

\input{sec/009Conclusion}
\input{sec/099Ethics}

\bibliography{references}
\bibliographystyle{plainnat}


\newpage
\input{sec/900TableOfContents}

\newpage
\appendix
\section*{\LARGE APPENDICES}
\input{sec/901Rescaling}
\newpage
\input{sec/902HardInstanceProofs}
\newpage
\input{sec/903NonAnytimeProofs}
\newpage
\input{sec/904AnytimeProofs}
\newpage
\input{sec/905TopK}

\end{document}

%% file: _header.tex
\usepackage[margin = 0.9in]{geometry}
\usepackage{fancyhdr}
\input{math_commands.tex}   

\usepackage[parfill]{parskip}
\usepackage[sort]{natbib}
\usepackage[pagebackref]{hyperref}  
\usepackage{url}    
\usepackage{booktabs}   
\usepackage[nopatch=footnote]{microtype}    
\usepackage[x11names,table,xcdraw,dvipsnames]{xcolor} 
\definecolor{darkblue}{rgb}{0.0,0.0,0.65}
\definecolor{darkred}{rgb}{0.65,0.0,0.0}
\definecolor{darkgreen}{rgb}{0.0,0.5,0.0}
\hypersetup{
	colorlinks = true,
	citecolor  = darkblue,
	linkcolor  = darkred,
	urlcolor   = darkgreen,
}

\usepackage[normalem]{ulem}  
\usepackage{enumerate}  
\usepackage{enumitem}   
\usepackage{graphicx}   
\usepackage{caption}
\usepackage{subcaption} 
\usepackage{multirow}   
\usepackage{multicol}   
\usepackage{etoc}   
\usepackage[ruled,vlined]{algorithm2e}  
\usepackage{setspace}   
\usepackage{tikz}   
\usetikzlibrary{arrows.meta,calc,matrix}
\usepackage{pgfplots}   
\usepgfplotslibrary{groupplots}
\pgfplotsset{compat=1.18}   

\usepackage[capitalize]{cleveref}   
\crefname{section}{Sec.}{Secs.}
\crefname{appendix}{Appx.}{Appcs.}
\crefname{equation}{Eq.}{Eqs.}
\crefname{definition}{Def.}{Defs.}
\crefname{figure}{Fig.}{Figs.}
\crefname{tabular}{Tab.}{Tabs.}
\crefname{table}{Tab.}{Tabs.}
\crefname{algorithm}{Alg.}{Algs.}
\crefname{theorem}{Thm.}{Thms.}
\crefname{lemma}{Lem.}{Lems.}
\crefname{proposition}{Prop.}{Props.}
\crefname{corollary}{Cor.}{Cors.}
\crefname{assumption}{Assum.}{Assums.}
\crefname{remark}{Rmk.}{Rmks.}

\theoremstyle{plain}
\newtheorem{theorem}{Theorem}[section]

\newtheorem{lemma}[theorem]{Lemma}

\theoremstyle{definition}

\theoremstyle{remark}
\newtheorem{remark}[theorem]{Remark}  

\title{\bfseries Stronger Lower Bounds for (Non-)Anytime\\
Acceleration of Gradient Descent}

\author{
    Minchan Jung\thanks{Authors contributed equally to this paper.} \\
    Korea Science Academy of KAIST \\
    \texttt{mickey080929@gmail.com} \\ \\
    \and
    Hanseul Cho$^*$, Chulhee Yun \\
    Graduate School of AI, KAIST \\
    \texttt{\string{jhs4015,chulhee.yun\string}@kaist.ac.kr} \\
}

%% file: math_commands.tex
\usepackage{amsmath,amsfonts,amsthm,amssymb,mathrsfs,bm,bbm,pifont,mathtools,mathdots,nicefrac,thmtools,thm-restate}

\def\norm#1{\left\lVert#1\right\rVert}
\def\inner#1{\left\langle#1\right\rangle}
\def\abs#1{\left|#1\right|}

\def\open#1{\left(#1\right)}
\def\bigset#1{\left\{#1\right\}}  
\def\closed#1{\left[#1\right]}

\let\bar\overline
\let\emptyset\varnothing

\def\Fclass{\mathscr{F}}

\def\1{\bm{1}}

\def\vzero{{\bm{0}}}

\def\ve{{\bm{e}}}

\def\vv{{\bm{v}}}

\def\vx{{\bm{x}}}
\def\vy{{\bm{y}}}

\def\veta{{\bm{\eta}}}

\def\vxi{{\bm{\xi}}}

\def\mI{{\bm{I}}}

\def\mM{{\bm{M}}}

\DeclareMathAlphabet{\mathsfit}{\encodingdefault}{\sfdefault}{m}{sl}
\SetMathAlphabet{\mathsfit}{bold}{\encodingdefault}{\sfdefault}{bx}{n}

\def\gB{{\mathcal{B}}}

\def\gG{{\mathcal{G}}}

\def\gM{{\mathcal{M}}}
\def\gN{{\mathcal{N}}}

\def\gP{{\mathcal{P}}}
\def\gQ{{\mathcal{Q}}}
\def\gR{{\mathcal{R}}}

\def\gU{{\mathcal{U}}}
\def\gV{{\mathcal{V}}}

\def\sN{{\mathbb{N}}}

\newcommand{\R}{\mathbb{R}}

\DeclareMathOperator*{\argmin}{arg\,min}

%% file: sec/000Abstract.tex
\begin{abstract}
The rate-optimal convergence rate of gradient descent (GD) with a fixed step-size is well known to be $\Theta(N^{-1})$ for $L$-Lipschitz smooth convex objectives in the prior art in convex optimization.
Surprisingly, several recent works show that we can accelerate vanilla GD by applying a nonconstant, nonadaptive, deterministic step-size schedule. 
The best-known \textbf{upper bounds} so far in the \emph{non-anytime} \& \emph{anytime} setups are $O(N^{-1.271})$~\citep{altschuler2025acceleration,grimmer2025accelerated,grimmer2025composing,zhang2026accelerated} and 
$O(N^{-1.119})$~\citep{zhang2025anytime}, respectively.
On the other hand, the best reported lower bounds (or barriers) up to date in the \emph{non-anytime} \& \emph{anytime} setups are $\Omega(N^{-1.635})$ and $\Omega(N^{-1.241})$~\citep{ye2026improved}, respectively.
We narrow these gaps by establishing \textbf{stronger lower bounds} for GD's convergence rate in both settings: $\Omega(N^{-1.450})$ for the non-anytime rate bound and $\Omega(N^{-1.184})$ for the anytime rate barrier.
\end{abstract}

%% file: sec/001Introduction.tex
\section{Introduction} 
\label{sec:intro}

We consider an unconstrained minimization problem of a real-valued differentiable function $f$ defined on a finite-dimensional Euclidean space $\R^d$, which is convex and whose gradient is $L$-Lipschitz (i.e., $f$ is $L$-smooth) for some $L>0$.
A fundamental optimization method to approximate a minimum of $f$ is \textit{gradient descent}, or GD for short~\citep{cauchy1847methode,goldstein1962cauchy,polyak1963gradient}.
Given an initial point $\vx_0\in \R^d$ and a sequence $(\eta_t)_{t\in \sN}$ of positive step-sizes (or, \textit{step-size schedule}), it is described as
\begin{align}
    \vx_t = \vx_{t-1} - \eta_t \nabla f (\vx_{t-1}), \qquad t \in \sN. \tag{GD} \label{eq:gd}
\end{align}
We study the \textbf{convergence rate} of GD in terms of the function value gap $f(\vx) - \min_{\vx' \in \R^d} f(\vx')$.
Let $\Fclass_L(\R^d)$ be the class of convex $L$-smooth functions $f:\R^d\to\R$ where the set of its minimizers $X^\star_f :=\argmin f$ satisfies $\emptyset \neq X^\star_f \subsetneq \R^d$.
Then, we define the convergence rate $\gR^{L}_{N}(\veta)$ associated with the number of steps $N\in \sN$, a step-size schedule $\veta = (\eta_1, \dotsc, \eta_N) \in (0,\infty)^N$, and a Lipschitz-smoothness parameter $L>0$ as
\begin{align}
    \gR^{L}_{N} (\veta)
    = \sup_{\substack{d\in \sN,\ f \in \Fclass_{L}(\R^d),\\
    \vx^\star \in X^\star_{f},\ \vx_{0} \in \R^d\setminus X^\star_{f}}} \frac{f(\vx_{N}) - f(\vx^\star)}{L \norm{\vx_{0} - \vx^\star}^2}.
    \label{eq:convergence_rate_R_N}
\end{align}
Without loss of generality, it suffices to study the case $L=1$, since $\gR^L_N(\veta) = \gR^1_N(L\cdot\veta)$, where $L\cdot\veta := (L \eta_t)_{t\in [N]}$: see \cref{sec:rescaling} for discussion.
Thus, we write and study $\gR_N(\veta):= \gR^{1}_N(\veta)$.

Traditional convex optimization literature has focused on small step-sizes $\eta_t \in (0, \frac{2}{L})$, especially which are constant over steps (usually $\eta_t = \frac{1}{L}$).
In such a case, the convergence rate of GD after $N$ iterations is known to be $\Theta(N^{-1})$~\citep{levitin1966constrained,drori2014performance,bubeck2015convex,nesterov2018smooth}.
This rate cannot improve under small step-size schedules of GD. 
The traditional literature has studied modifications to optimization algorithms, such as Nesterov's momentum~\citep{nesterov1983method}, to obtain a tight convergence rate that matches the $\Omega(N^{-2})$ lower bound for more general first-order algorithms~\citep{nemirovsky1983problem}.\footnote{The lower bound considers the class of first-order algorithms which adds a linear combination of gradients at current and previous iterates to the current iterate at every iteration. It includes GD and (Nesterov/Polyak) momentum GD.}
However, the step-size schedule itself has not been recognized as a tool for accelerating GD until recently.

\vspace{-5pt}
\subsection{Step-Size Schedules for Accelerating GD}
\label{sec:intro_relatedwork}
\vspace{-5pt}

Surprisingly, recent studies have highlighted that a carefully designed step-size schedule can achieve an accelerated rate~\citep{altschuler2025acceleration,grimmer2023accelerated,grimmer2024provably,grimmer2025accelerated,grimmer2025composing,zhang2025anytime,zhang2026accelerated}, proven to be strictly faster than the classical $\Theta(N^{-1})$ rate.
It departs from the constant-stepsize assumption and occasionally uses step-sizes much larger than $\frac{2}{L}$.
A notable example is the \emph{silver step-size schedule}~\citep{altschuler2025acceleration}: in the \textbf{non-anytime} setting with a known terminal iteration (of the form $N=2^a-1$ for a positive integer $a\in\sN$), it attains the upper bound $O(N^{-\alpha_{\mathrm{sil}}}) \approx O(N^{-1.271})$.%
\footnote{The exponent is defined as $\alpha_{\mathrm{sil}}:= \log_2 \rho_{\mathrm{sil}}=1.2715\cdots$, where $\rho_{\mathrm{sil}}:= 1+\sqrt{2} = 2.4142\ldots$ is often called the \emph{silver ratio}.
Observe that rounding $\alpha_{\mathrm{sil}}$ to $1.272$ will incorrectly write a convergence upper bound strictly faster than $O(N^{-\alpha_{\mathrm{sil}}})$.
Thus, to ensure a conservative yet valid statement, we round down (up, resp.) the absolute value of the negative exponent when reporting a convergence rate upper bound (lower bound, resp.).}
A few concurrent and subsequent works have also proved the same $O(N^{-\alpha_{\mathrm{sil}}})$ rate, while some of them have improved constant factors by applying concatenation and composition techniques~\citep{grimmer2025accelerated,grimmer2025composing,zhang2026accelerated}.
For the \textbf{anytime} setting, where the terminal iteration $N$ is not fixed before designing a step-size schedule, \citet{zhang2025anytime} proposed a step-size schedule that attains a convergence rate upper bound of $O(N^{-2\alpha_{\mathrm{sil}} / (1+\alpha_{\mathrm{sil}})}) \approx O(N^{-1.119})$.
However, it remains an open question whether an alternative step-size schedule can improve these upper-bound results. 

More recently, a series of lower-bound results has also been reported, narrowing the gap between the upper bounds and the classical $\Omega(N^{-2})$ lower bound by~\citet{nemirovsky1983problem}.
In the \textit{anytime} setting, \citet{tsai2026lower} rules out $\gR_N(\veta) = o(N^{-4/3}) \approx o(N^{-1.334})$ for any nonadaptive\footnote{A step-size schedule is said to be \textit{nonadaptive} if it is fixed before the algorithm is run; that is, it does not vary in response to intermediate quantities (e.g., $\vx_t$, $f(\vx_t)$, and $\nabla f(\vx_t)$) generated while it runs.} deterministic\footnote{A step-size schedule is said to be \textit{deterministic} if running the algorithm with that schedule from the same initial point always produces the same outcome.} step-size schedule.
Their analysis builds on their analysis of large step-sizes and partial sums of step-sizes on a 1-dimensional quadratic and asymmetric Huber function.
On the other hand, extending the domain's dimension for the hard instance function and selecting multiple large steps, \citet{ma2026lower} and \citet{tsai2026improved} later establish \textit{non-anytime} rate lower bounds of $\Omega(N^{-\sqrt{2+\sqrt{3}}}) \approx \Omega(N^{-1.932})$ and $\Omega(N^{-\sqrt{3}}) \approx \Omega(N^{-1.733})$, respectively. 
Very recently, \citet{ye2026improved} report even more improved results of $\Omega(N^{-1.635})$ non-anytime rate bound and $\Omega(N^{-1.241})$ anytime rate barrier.
Readers may refer to \cref{fig:summary_GD_timeline} for a summary of these results.

\subsection{Our Contributions}
\label{sec:intro_contributions}

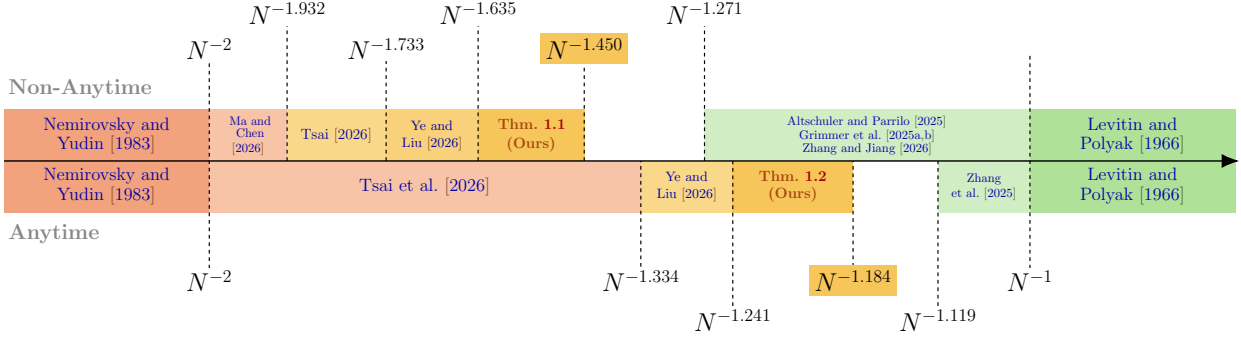
\begin{figure}[t]
\centering
\resizebox{\linewidth}{!}{\input{fig/gd_lowerbound_timeline.tikz}}
\caption{A summary of (non-)anytime convergence rate upper/lower bounds of GD on smooth convex minimization. We slightly modified Figure~1 by~\citet{tsai2026improved}.
The x-axis is not perfectly scaled linearly in exponents.}\label{fig:summary_GD_timeline}
\end{figure}

In this work, we establish improved lower bounds on the convergence rate of GD for smooth convex minimization in both non-anytime and anytime settings.
We summarize our contributions as the following two main theorems. See \cref{thm:NonAnytimeLowerBoundFull,thm:AnytimeLowerBoundFull} for their full statements.

\begin{restatable}[Non-anytime lower bound, informal]{theorem}{thmNonAnytimeLowerBound}
\label{thm:NonAnytimeLowerBound}
Fix any $N\!\ge\!1$. For any $N$-step schedule
$\veta\in(0,\infty)^N$, it holds that $\gR_N(\veta) = \Omega(N^{-\alpha_\star\!})$ with  $\alpha_\star := \log_2(1+\sqrt{\color{red}3}) = 1.4499\cdots$.
\end{restatable}

\begin{restatable}[Anytime lower bound, informal]{theorem}{thmAnytimeLowerBound}
\label{thm:AnytimeLowerBound}
There is a universal constant $C_\star\!>\!0$ such that, for any infinite schedule
$\veta := (\eta_1,\eta_2,\eta_3,\cdots) \in(0,\infty)^{\sN}$, it holds that $\limsup_{n\to\infty} n^{\beta_\star} \gR_n (\veta_{1:n}) \ge C_\star$ with  $\beta_\star := \frac{2\alpha_\star}{1 + \alpha_\star} = 1.1836\cdots$.
As a result, \textbf{no} positive infinite schedule satisfies $\gR_n (\veta_{1:n}) = o(n^{-\beta_\star})$.
\end{restatable}

\begin{remark}
    As far as we know, \cref{thm:AnytimeLowerBound} is one of the \emph{first} results that implies the anytime convergence rate in terms of $\gR_N(\veta)$ must be strictly slower than the non-anytime rate ($\because \beta_\star < \alpha_{\mathrm{sil}}$) when \emph{naively} comparing the convergence rate exponents.
    Equivalently, following the discussion by~\citet{kornowski2024open}, it implies that the anytime
    rate in terms of \emph{individual iterates} is provably slower than the anytime rate in terms of the \emph{best iterate} (naturally implied by the non-anytime rate upper bound).
\end{remark}

\paragraph{Notation.}
We use superscripts with parentheses to write coordinate indices.
For a vector $\vx \!=\! \open{x^{(1)},\dotsc,x^{(d)}} \in \R^d$, we set $x^{(d')} \!=\! 0$ if $d' > d$ when needed.
Let $\sN$ be the set of all positive integers and
$[n]:=\{1,\dotsc,n\}$ for $n\in\sN$.
For positive sequences $(a_n)_{n\in\sN}$ and $(b_n)_{n\in\sN}$, write $a_n \!=\! O(b_n)$ and/or $b_n \!=\! \Omega(a_n)$ if $\exists c>0$ s.t. $a_n \!\le\! c\, b_n$ for large enough $n\!\ge\! 1$ (thus, $\limsup_{n\to\infty} a_n/b_n \!<\! \infty$).
Write $a_n \!=\! \Theta(b_n)$ if $a_n \!=\! O(b_n)$ and $a_n \!=\! \Omega(b_n)$.
Write $a_n \!=\! o(b_n)$ and/or $b_n \!=\! \omega(a_n)$ if $\lim_{n\to\infty} a_n/b_n \!=\! 0$.
For $\veta = (\eta_n)_{n\in\sN}$, we write $\veta_{m:n} := (\eta_m,\dotsc,\eta_n)$ if $1 \!\le\! m \!\le\! n$, whereas $\veta_{m:n}$ is an empty sequence if $m>n$.
We often write $\{t_1 < \dotsb < t_k\}$ to denote a strictly increasing sequence $(t_i)_{i\in [k]}$.
For each $d\in \sN$, we equip $\R^d$ with a Euclidean ($\ell_2$) norm $\norm{\cdot}$ and a standard inner product $\inner{\cdot,\cdot}$.
When the ambient dimension $d$ is clear from the context, write $\ve_i$ as the $i$-th unit vector in $\R^d$: all components are zero, except for a single 1 at its $i$-th component.
For a differentiable function $f: \R^d \to \R$, we say it is convex if $f(\vx) \!\ge\! f(\vy) \!+\! \inner{\nabla f(\vy), \vx \!-\! \vy}$ ($\forall \vx,\vy \in \R^d$); for $L>0$, we say $f$ is $L$-smooth if $\norm{\nabla f(\vx) \!-\! \nabla f(\vy)} \!\le\! L \norm{\vx \!-\! \vy}$ ($\forall \vx,\vy \in \R^d$).

%% file: fig/gd_lowerbound_timeline.tikz

\definecolor{lowerA}{HTML}{F2A37E}
\definecolor{lowerB}{HTML}{F7C4A8}
\definecolor{lowerC}{HTML}{F9D88A}
\definecolor{lowerD}{HTML}{F8D17A}
\definecolor{ours}{HTML}{F6C65B}
\definecolor{upperA}{HTML}{D1EDC4}
\definecolor{upperB}{HTML}{AFE198}
\definecolor{lowerText}{HTML}{7A3A1D}
\definecolor{oursText}{HTML}{A8470D}
\definecolor{upperText}{HTML}{315F23}
\definecolor{rowLabel}{HTML}{888888}

\begin{tikzpicture}[x=1cm,y=1cm,>=Latex]

  \def\ytop{1.45}
  \def\ybot{-1.45}

  \def\xL{-2.80}
  \def\xTwo{3.00}          
  \def\xMa{5.20}           
  \def\xTsaiB{8.00}        
  \def\xYeNonAny{10.60}    
  \def\xOursOne{13.60}     
  \def\xTsaiA{15.20}       
  \def\xUpperNonAny{17.00} 
  \def\xYeAny{17.80}       
  \def\xOursTwo{21.20}     
  \def\xUpperAny{23.60}    
  \def\xOne{26.20}         
  \def\xR{32.00}

  \fill[lowerA] (\xL,0) rectangle (\xTwo,\ytop);
  \fill[lowerB] (\xTwo,0) rectangle (\xMa,\ytop);
  \fill[lowerC] (\xMa,0) rectangle (\xTsaiB,\ytop);
  \fill[lowerD] (\xTsaiB,0) rectangle (\xYeNonAny,\ytop);
  \fill[ours]   (\xYeNonAny,0) rectangle (\xOursOne,\ytop);

  \fill[upperA] (\xUpperNonAny,0) rectangle (\xOne,\ytop);
  \fill[upperB] (\xOne,0) rectangle (\xR,\ytop);

  \fill[lowerA] (\xL,\ybot) rectangle (\xTwo,0);
  \fill[lowerB] (\xTwo,\ybot) rectangle (\xTsaiA,0);
  \fill[lowerC] (\xTsaiA,\ybot) rectangle (\xYeAny,0);
  \fill[ours]   (\xYeAny,\ybot) rectangle (\xOursTwo,0);

  \fill[upperA] (\xUpperAny,\ybot) rectangle (\xOne,0);
  \fill[upperB] (\xOne,\ybot) rectangle (\xR,0);

  \draw[-{Latex[length=5mm,width=4mm]},line width=1.15pt]
    (\xL,0) -- (\xR+0.12,0);

  \node[
    anchor=west,
    text=rowLabel,
    font=\fontsize{17.28}{20}\selectfont\bfseries 
  ] at (\xL,2.05) {Non-Anytime};

  \node[
    anchor=west,
    text=rowLabel,
    font=\fontsize{17.28}{20}\selectfont\bfseries 
  ] at (\xL,-2.05) {Anytime};

  \node[
    align=center,
    text=lowerText,
    text width=0.9*(\xTwo-\xL)*1cm,
    font=\fontsize{14}{16}\selectfont
  ] at ({(\xL+\xTwo)/2},0.72)
    {\citet{nemirovsky1983problem}};

  \node[
    align=center,
    text=lowerText,
    text width=0.9*(\xMa-\xTwo)*1cm,
    inner sep=0pt,
    font=\fontsize{10}{11.5}\selectfont
  ] at ({(\xTwo+\xMa)/2},0.72)
    {\makeatletter
     \def\NAT@nmfmt#1{\shortstack{Ma and\\Chen}}
     \citet{ma2026lower}
     \makeatother};

  \node[
    align=center,
    text=lowerText,
    text width=0.9*(\xTsaiB-\xMa)*1cm,
    font=\fontsize{12}{14.4}\selectfont
  ] at ({(\xMa+\xTsaiB)/2},0.72)
    {\citet{tsai2026improved}};

  \node[
    align=center,
    text=lowerText,
    text width=0.9*(\xYeNonAny-\xTsaiB)*1cm,
    inner sep=0pt,
    font=\fontsize{11}{13}\selectfont
  ] at ({(\xTsaiB+\xYeNonAny)/2},0.72)
    {\citet{ye2026improved}};

  \node[
    align=center,
    text=oursText,
    text width=0.9*(\xOursOne-\xYeNonAny)*1cm,
    font=\fontsize{12}{15}\selectfont\bfseries
  ] at ({(\xYeNonAny+\xOursOne)/2},0.72)
    {\cref{thm:NonAnytimeLowerBound}\\(Ours)};

  \node[
    align=center,
    text=upperText,
    text width=0.9*(\xOne-\xUpperNonAny)*1cm,
    font=\fontsize{10}{10.8}\selectfont
  ] at ({(\xUpperNonAny+\xOne)/2},0.72)
    {\citet{altschuler2025acceleration}\\\citet{grimmer2025accelerated,grimmer2025composing}\\\citet{zhang2026accelerated}};

  \node[
    align=center,
    text=upperText,
    text width=0.9*(\xR-\xOne)*1cm,
    font=\fontsize{14.0}{16}\selectfont
  ] at ({(\xOne+\xR)/2},0.72)
    {\citet{levitin1966constrained}};

  \node[
    align=center,
    text=lowerText,
    text width=0.9*(\xTwo-\xL)*1cm,
    font=\fontsize{14}{16}\selectfont
  ] at ({(\xL+\xTwo)/2},-0.72)
    {\citet{nemirovsky1983problem}};

  \node[
    align=center,
    text=lowerText,
    text width=0.9*(\xTsaiA-\xTwo)*1cm,
    font=\fontsize{14}{16}\selectfont
  ] at ({(\xTwo+\xTsaiA)/2},-0.72)
    {\citet{tsai2026lower}};

  \node[
    align=center,
    text=lowerText,
    text width=0.9*(\xYeAny-\xTsaiA)*1cm,
    font=\fontsize{11}{13}\selectfont
  ] at ({(\xTsaiA+\xYeAny)/2},-0.72)
    {\citet{ye2026improved}};

  \node[
    align=center,
    text=oursText,
    text width=0.9*(\xOursTwo-\xYeAny)*1cm,
    font=\fontsize{12}{14}\selectfont\bfseries
  ] at ({(\xYeAny+\xOursTwo)/2},-0.72)
    {\cref{thm:AnytimeLowerBound}\\(Ours)};

  \node[
    align=center,
    text=upperText,
    text width=0.9*(\xOne-\xUpperAny)*1cm,
    font=\fontsize{10}{12}\selectfont
  ] at ({(\xUpperAny+\xOne)/2},-0.72)
    {\citet{zhang2025anytime}};

  \node[
    align=center,
    text=upperText,
    text width=0.9*(\xR-\xOne)*1cm,
    font=\fontsize{14.0}{16}\selectfont
  ] at ({(\xOne+\xR)/2},-0.72)
    {\citet{levitin1966constrained}};


  \draw[dashed,line width=0.95pt]
    (\xTwo,-3.25) -- (\xTwo,3.05);

  \draw[dashed,line width=0.95pt]
    (\xMa,0) -- (\xMa,4.05);

  \draw[dashed,line width=0.95pt]
    (\xTsaiB,0) -- (\xTsaiB,3.05);

  \draw[dashed,line width=0.95pt]
    (\xYeNonAny,0) -- (\xYeNonAny,4.05);

  \draw[dashed,line width=0.95pt]
    (\xOursOne,0) -- (\xOursOne,3.05);

  \draw[dashed,line width=0.95pt]
    (\xUpperNonAny,0) -- (\xUpperNonAny,4.05);

  \draw[dashed,line width=0.95pt]
    (\xOne,-3.25) -- (\xOne,3.05);

  \draw[dashed,line width=0.95pt]
    (\xTsaiA,-3.20) -- (\xTsaiA,0);

  \draw[dashed,line width=0.95pt]
    (\xYeAny,-4.10) -- (\xYeAny,0);

  \draw[dashed,line width=0.95pt]
    (\xOursTwo,-3.20) -- (\xOursTwo,0);

  \draw[dashed,line width=0.95pt]
    (\xUpperAny,-4.10) -- (\xUpperAny,0);


  \node[
    fill=white,
    inner sep=1.5pt,
    font=\fontsize{20.74}{22}\selectfont
  ] at (\xTwo,3.18)
    {$N^{-2}$};

  \node[
    fill=white,
    inner sep=1.5pt,
    font=\fontsize{20.74}{22}\selectfont
  ] at (\xMa,4.18)
    {$N^{-1.932}$};

  \node[
    fill=white,
    inner sep=1.5pt,
    font=\fontsize{20.74}{22}\selectfont
  ] at (\xTsaiB,3.18)
    {$N^{-1.733}$};

  \node[
    fill=white,
    inner sep=1.5pt,
    font=\fontsize{20.74}{22}\selectfont
  ] at (\xYeNonAny,4.18)
    {$N^{-1.635}$};

  \node[
    fill=ours,
    inner sep=4pt,
    font=\fontsize{20.74}{22}\selectfont
  ] at (\xOursOne,3.18)
    {$N^{-1.450}$};

  \node[
    fill=white,
    inner sep=1.5pt,
    font=\fontsize{20.74}{22}\selectfont
  ] at (\xUpperNonAny,4.18)
    {$N^{-1.271}$};

  \node[
    fill=white,
    inner sep=1.5pt,
    font=\fontsize{20.74}{22}\selectfont
  ] at (\xTwo,-3.38)
    {$N^{-2}$};

  \node[
    fill=white,
    inner sep=1.5pt,
    font=\fontsize{20.74}{22}\selectfont
  ] at (\xTsaiA,-3.38)
    {$N^{-1.334}$};

  \node[
    fill=white,
    inner sep=1.5pt,
    font=\fontsize{20.74}{22}\selectfont
  ] at (\xYeAny,-4.5)
    {$N^{-1.241}$};

  \node[
    fill=ours,
    inner sep=4pt,
    font=\fontsize{20.74}{22}\selectfont
  ] at (\xOursTwo,-3.38)
    {$N^{-1.184}$};

  \node[
    fill=white,
    inner sep=1.5pt,
    font=\fontsize{20.74}{22}\selectfont
  ] at (\xUpperAny,-4.5)
    {$N^{-1.119}$};

  \node[
    fill=white,
    inner sep=1.5pt,
    font=\fontsize{20.74}{22}\selectfont
  ] at (\xOne,-3.38)
    {$N^{-1}$};

\end{tikzpicture}

%% file: sec/002HardInstance.tex
\vspace{-5pt}
\section{Hard Instance Construction Based on Checkpoint Selection}
\label{sec:hard_instance}
\vspace{-5pt}

Fix any $N\in \sN$ and a finite step-size schedule $\veta = (\eta_t)_{t\in [N]}$.
In this section, we aim to construct a hard instance function $F$ that depends on $\veta$ and establish a general lower bound for $\gR_N(\veta)$.
To explain how we construct $F$, fix any $k \in \{0\} \cup [N]$.
Intuitively, we divide $\veta$ into $k+1$ blocks by selecting $k$ \emph{checkpoint timesteps}. Then, we assign a \emph{component function} to each block, designed to activate sequentially: only one new component function activates in each block.
The idea of dividing a step-size schedule into several chunks has also been independently applied in a concurrent work by \cite{ma2026lower}, although our construction of a hard instance is completely different from theirs.

Choose $k$ checkpoint timesteps $T = \{t_1< \dotsb <t_k\}\subseteq [N]$, and set $t_0=0$ and $t_{k+1}=N+1$.
These checkpoints partition $N$ step-sizes into $k+1$ blocks.
Define
\begin{align}
\begin{aligned}
    \text{\emph{gap} ($t_{i-1} < t <t_{i}$):}& \quad
    s_i(T;\veta)
    = \textstyle\sum_{\tau = t_{i-1} + 1}^{t_i - 1} \eta_{\tau},
    &&i\in [k+1],  \\
    \text{\emph{checkpoint} ($t=t_{i}$):}& \quad
    b_i(T;\veta)
    = \eta_{t_i},
    &&i\in [k]. 
\end{aligned} \label{eq:gap_and_checkpoint}
\end{align}
Thus, for $i\in [k]$, the block $i$ consists of a \emph{checkpoint} step-size $b_i$ and a \emph{gap} preceding it (with sum $s_i$).
The terminal block  ($i=k+1$) consists only of the final gap $s_{k+1}$.
We omit $T$ and/or $\veta$ from the notation when clear from the context.
When $k=0$, we simply let $s_1 = \sum_{\tau=1}^N \eta_\tau$.

We associate one component function with each block and construct the hard instance $F:\R^{k+1}\to\R$ as their sum.
During each block, at most 1 new component function produces a nonzero gradient, and the checkpoint step $b_i$ activates the $(i+1)$-th component.
Hence, the selected checkpoints are exactly the transitions between successive 
active components.

For each component, we use a one-sided Huber function $H_{\delta}(\cdot)$ with $\delta>0$, defined as
\begin{equation}
    H_\delta (z) = \begin{cases}
    \begin{aligned}
        &0, && \text{if } z \le 0; &&\text{\color{lightgray}(inactive; zero gradient)}\\
        &\tfrac12 z^2, && \text{if } 0 \le z \le \delta; &&\text{\color{lightgray}(quadratic region)}\\
        &\delta z - \tfrac{\delta^2}{2}, &&\text{if } z \ge \delta.  &&\text{\color{lightgray}(nonzero affine region)}
    \end{aligned}
    \end{cases}  \label{eq:one_sided_huber}
\end{equation}
Its derivative $H'_\delta (z) = \min\{\max\{0, z\}, \delta\}$ is nondecreasing and 1-Lipschitz. Hence, $H_\delta$ is nonnegative, convex, and 1-smooth. 
We also define $i$-th \emph{margin} $z_i(\vx)$ as
\begin{align*}
    z_i(\vx) := x^{(i)} - x^{(i+1)} - c_i \quad (i=1,\cdots,k),\qquad z_{k+1}(\vx) := x^{(k+1)} - c_{k+1}, 
\end{align*}
for some constants $c_i \ge 0$ ($i \in [k+1]$) to be determined later.
We construct the hard instance as
\begin{align}
    F(\vx) := \frac14 \sum_{i=1}^k H_{\delta_i}\open{z_i(\vx)} + \frac12 H_{\delta_{k+1}}\open{z_{k+1}(\vx)}, \label{eq:hard_instance}
\end{align}
for some constants $\delta_i > 0$ ($i \in [k+1]$) to be determined later. 

The following proposition proves some properties of $F$; namely, $F\in \Fclass_1(\R^{k+1})$, $F\ge 0$, and $\vzero \in X_F^\star$. This is why this function is sufficient for analyzing $\gR_N(\veta)$. 
\begin{restatable}{proposition}{propHardInstanceProperties}
\label{prop:HardInstanceProperties}
    The function $F$ (\cref{eq:hard_instance}) is nonnegative, convex, and 1-smooth. 
    Moreover, the origin $\vzero$ is a minimizer since $F(\vzero) = 0$.
    \end{restatable}
As a result of \cref{prop:HardInstanceProperties}, for a given finite schedule $\veta\in(0, \infty)^N$, we have
\begin{align}
   \gR_N(\veta) \ge \frac{F(\vx_N)-F(\vzero)}{\norm{\vx_0-\vzero}_2^2} = \frac{F(\vx_N)}{\norm{\vx_0}_2^2}. \label{eq:R_N_to_funcval}
\end{align}
Thus, it suffices to lower-bound the right-hand side of this equation.

Next, we present a key technical lemma about one-step GD dynamics on $F$ as below.
\begin{restatable}[One-step dynamics]{lemma}{lemOneStepGD}
\label{lem:OneStepGD}
    Fix $\vx\in \R^{k+1}$ and let $\vx_+ := \vx - \eta \nabla F(\vx)$ for $\eta>0$.
    \begin{enumerate}[label=(\roman*),leftmargin=*]
        \item \label{item:margin_one_step_lower_bound} For each $i\in [k+1]$, if $z_i(\vx) \ge \delta_i$, then $z_i(\vx_+) \ge z_i(\vx) - \frac{\eta \delta_i}{2}$.
        \item \label{item:margin_one_step_edge} For each $i\in[k]$, if $z_i(\vx) \ge \delta_i$ and $z_{j}(\vx) \le 0$ for all $j = i+1, \cdots, k+1$, then 
        \begin{align*}
            x_+^{(i+1)} = x^{(i+1)} + \frac{\eta \delta_i}{4}, \qquad x_+^{(j')} = x^{(j')} \quad (j'\ge i+2).
        \end{align*}
    \end{enumerate}
\end{restatable}
We defer the proofs of \cref{prop:HardInstanceProperties,lem:OneStepGD} to \cref{subsec:hard_instance_properties} and \cref{subsec:one_step_dynamics}.
Observe that they hold regardless of the choice of parameters $c_i\ge 0$ and $\delta_i>0$.

Based on these ingredients, we now establish a general lower bound on $\gR_N(\veta)$ in terms of the block format of $\veta \mapsto (s_1, b_1, \dotsc,s_k,b_k,s_{k+1})$ due to the checkpoint timesteps $T=\bigset{t_1<\dots<t_k}$.
\begin{restatable}[General lower bound]{lemma}{lemGeneralLowerBound}
\label{lem:GeneralLowerBound}
    Fix any checkpoint timesteps $T=\bigset{t_1<\dots<t_k}\subseteq[N]$, gaps $s_i:=s_i(T;\veta)$, and checkpoint steps $b_i:=b_i(T;\veta)$: see \cref{eq:gap_and_checkpoint}.
    Then, we have a lower bound on the GD convergence rate, defined in \cref{eq:convergence_rate_R_N}, as
    \begin{align}
        \mathcal R_N(\veta)
        \ge \frac{1}{4(1+s_{k+1})}\prod_{i=1}^k\left(\frac{b_i}{2(2+s_i)}\right)^2.
        \label{eq:general_normalized_lower_bound}
    \end{align}
\end{restatable}
\begin{proof}[Proof Sketch of~\cref{lem:GeneralLowerBound}]
    Let us take an initial point $\vx_0 \!=\! r_0\ve_1$ ($r_0 > 0$).
    Since $\norm{\vx_0} \!=\! r_0\norm{\ve_1} \!=\! r_0$ and because of \cref{eq:R_N_to_funcval}, it suffices to show that 
    \begin{align}
        F(\vx_N)\ge
        \frac{r_0^2}{4(s_{k+1}+1)} \prod_{i=1}^k\left(\frac{b_i}{2(s_i+2)}\right)^2.
        \label{eq:general_function_value_lower_bound}
    \end{align}
    To this end, we set the parameters $c_i\ge0$ and $\delta_i>0$ recursively.
    Set $c_1=0$ and $A_1 = r_0$. Then, set
    \begin{align*}
        \delta_i = \frac{2A_i}{2 + s_i},~~
        c_{i+1} = \frac{s_i \delta_i}{4},~~
        A_{i+1} = \frac{b_i \delta_i}{4},~~ (\forall i\in[k]); \qquad  
        \delta_{k+1} = \frac{A_{k+1}}{1 + s_{k+1}}.
    \end{align*}
    In particular, we obtain  $A_i = r_0 \prod_{j=1}^{i-1} \frac{b_j}{2(2+s_j)}$ for all $i\in[k+1]$ by unrolling the recursion.
    
    \textbf{Claim.}~~At the beginning of $i$-th block ($\vx_{t_{i-1}}$), the newly defined parameter $A_i$ tracks the $i$-th margin at the checkpoint step (i.e., $z_i(\vx_{t_{i-1}}) = A_i\ge \delta_i$); thus, the $i$-th margin is in nonzero affine region of $H_{\delta_i}$. Moreover, the subsequent components ($j\ge i+1$) of the iterate are zero (i.e., $x_{t_{i-1}}^{(j)}=0$); thus, the subsequent margins are all inactive.
    
    We show the claim by induction on $i\in[k+1]$.
    Since $z_1(\vx_0)=x^{(1)}_0 - x^{(2)}_0 -c_1 = r_0 = A_1$ and $x^{(j)}_0=0$ ($j\ge2$), the claim holds at $i=1$.
    Also, if the claim holds at $i\in [k]$, we apply \cref{lem:OneStepGD} repeatedly to show that the claim also holds at $i+1$.
    This proves the claim up to $i\le k+1$.
    The core idea (roughly depicted in \cref{fig:one_block_dynamics}) is that:
    \begin{itemize}[leftmargin=*,nosep]
        \item (\cref{lem:GapsStayAffine})~~During the gap of the block $i$ ($t = t_{i-1} + 1, \dotsc,t_{i}-1$), the $i$-th margins of iterates remain in its nonzero affine region of $H_{\delta_i}$, while all subsequent margins remain inactive.
        \item (\cref{lem:BigStepActivatesNext})~~The checkpoint step $\eta_{t_i}$ activates $(i+1)$-th margin (i.e., $H_{\delta_{i+1}}$ has a nonzero derivative), thereby passing the same structure to the next block. All coordinate $j\ge i+2$ remain zero.
    \end{itemize}
    \begin{figure}[t]
    \centering
    \resizebox{\linewidth}{!}{\input{fig/one_block_dynamics.tikz}}
    \caption{Block-wise visualization of GD iterates on our hard instance (\cref{eq:hard_instance}). It  illustrates the trajectory of $(x_t^{(i)}, x_t^{(i+1)}, x_t^{(i+2)})$ during the block $i\in [k]$: $t=t_{i-1} \!+\! 1, \dotsc, t_{i} \!-\! 1, t_{i}$. We write the $i$-th margin $z_i=z_i(\vx_t)$ corresponding to $i$-th component function $H_{\delta_i}(\cdot)$. Note that $x_t^{(i+2)} \equiv 0$ until $t \le t_i$.}
    \label{fig:one_block_dynamics}
    \end{figure}
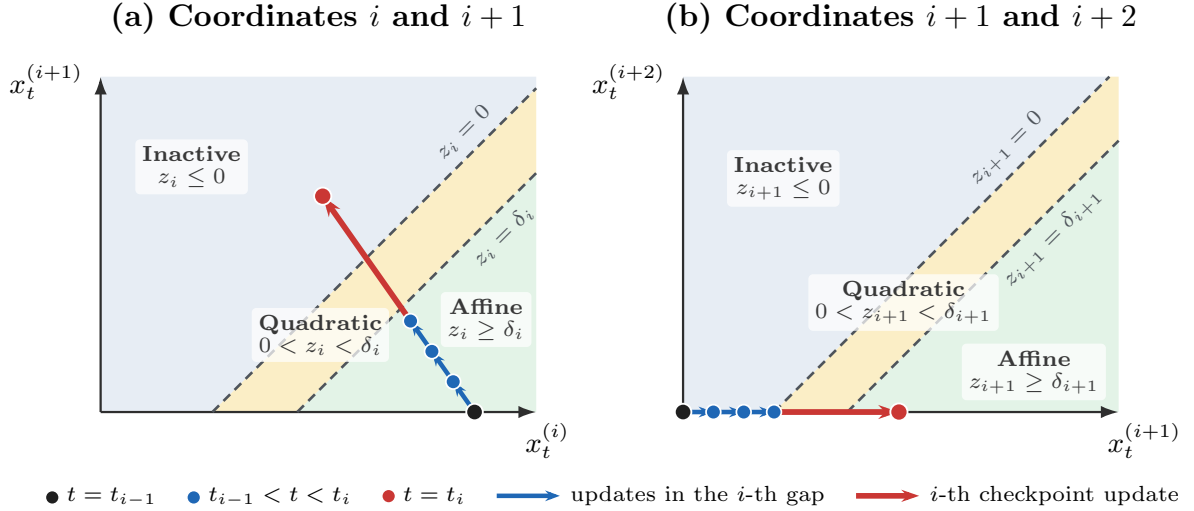
    In particular, we have $z_{k+1}(\vx_{t_k}) \!=\! A_{k+1} \!\ge\! \delta_{k+1}$, so we can apply \cref{lem:OneStepGD}~\ref{item:margin_one_step_lower_bound} for $(k+1)$-th margins.
    With a similar logic, by applying \cref{lem:OneStepGD}~\ref{item:margin_one_step_lower_bound} repeatedly, we can eventually prove that $z_{k+1}(\vx_N) \ge A_{k+1} - \frac{s_{k+1} \delta_{k+1}}{2}$ holds at the final iterate.

    For brevity, let us omit the subscript $k+1$ from now on: $A=A_{k+1}$, $s=s_{k+1}$, $\delta=\delta_{k+1}=A/(1+s)$.
    Observe that $A - \frac{s \delta}{2} = \delta\cdot (1+\frac{s}{2})\ge \delta$ and
    \begin{align*}
        F(\vx_N)
        &\ge \frac{1}{2}H_\delta\open{z_{k+1}(\vx_N)} 
        =\frac{\delta}{2}\, z_{k+1}(\vx_N)-\frac{\delta^2}{4}
        \ge \frac{\delta}{2}\left(A-\frac{s\delta}{2}\right)-\frac{\delta^2}{4} =\frac{A^2}{4(1+s)}.
    \end{align*}
    This proves \cref{eq:general_function_value_lower_bound}.
    See \cref{subsec:general_lowerbound_proof} for the full proof.
\end{proof}

%% file: fig/one_block_dynamics.tikz
\definecolor{smallblue}{RGB}{31,103,181}
\definecolor{selectedred}{RGB}{202,55,52}
\definecolor{boundarygray}{RGB}{78,84,91}
\definecolor{inactivefill}{RGB}{222,230,240}
\definecolor{quadraticfill}{RGB}{250,232,176}
\definecolor{affinefill}{RGB}{216,239,222}

\tikzset{
  region label/.style={
    align=center,
    font=\scriptsize,
    text=black!78,
    rounded corners=1.5pt,
    fill=white,
    fill opacity=0.72,
    text opacity=1,
    inner sep=2.2pt
  }
}

\pgfplotsset{
  one block axis/.style={
    width=6.75cm,
    height=5.35cm,
    xmin=0,
    xmax=5.65,
    ymin=0,
    ymax=4.35,
    axis equal image,
    set layers=standard,
    axis lines=left,
    axis line style={
      black!82,
      line width=0.72pt,
      -{Latex[length=2.0mm,width=1.4mm]}
    },
    ticks=none,
    enlargelimits=false,
    clip=false,
    title style={font=\bfseries, yshift=1.5mm},
    xlabel style={
      at={(axis description cs:0.95,0)},
      anchor=north west
    }
  },
  region fill/.style={
    on layer=axis background,
    draw=none,
    fill opacity=0.72,
    forget plot
},
  boundary plot/.style={
    boundarygray,
    densely dashed,
    line width=0.85pt,
    samples=2,
    forget plot
  },
  small segment/.style={
    smallblue,
    line width=1.45pt,
    -{Stealth[length=1.65mm,width=1.2mm]},
    forget plot
  },
  trajectory points/.style={
    only marks,
    mark=*,
    mark size=2.35pt,
    mark options={fill=smallblue, draw=white, line width=0.55pt},
    forget plot
  },
  selected trajectory/.style={
    selectedred,
    line width=1.85pt,
    -{Stealth[length=2.3mm,width=1.55mm]},
    forget plot
  },
  selected point/.style={
    only marks,
    mark=*,
    mark size=2.75pt,
    mark options={fill=selectedred, draw=white, line width=0.6pt},
    forget plot
  },
  first point/.style={
    only marks,
    mark=*,
    mark size=2.65pt,
    mark options={fill=black!88, draw=white, line width=0.6pt},
    forget plot
  }
}

\begin{tikzpicture}[font=\small]
\begin{groupplot}[
  group style={group size=2 by 1, horizontal sep=1.65cm},
  one block axis
]

  \nextgroupplot[
    title={(a) Coordinates $i$ and $i+1$},
    xlabel={$x_t^{(i)}$}
  ]

  \node[anchor=east]
    at (axis description cs:-0.015,0.985) {$x_t^{(i+1)}$};

  \addplot[region fill, fill=inactivefill]
    coordinates {(0,0) (1.45,0) (5.65,4.20) (5.65,4.35)
                 (0,4.35) (0,0)};
  \addplot[region fill, fill=quadraticfill]
    coordinates {(1.45,0) (2.55,0) (5.65,3.10) (5.65,4.20)
                 (1.45,0)};
  \addplot[region fill, fill=affinefill]
    coordinates {(2.55,0) (5.65,0) (5.65,3.10) (2.55,0)};

  \addplot[boundary plot, domain=1.45:5.65] {x-1.45}
    node[pos=0.82, sloped, above, font=\scriptsize, text=boundarygray]
      {$z_i=0$};
  \addplot[boundary plot, domain=2.55:5.65] {x-2.55}
    node[pos=0.80, sloped, below, font=\scriptsize, text=boundarygray]
      {$z_i=\delta_i$};

  \node[region label] at (axis cs:1.18,3.18)
    {\textbf{Inactive}\\[-1pt]$z_i\leq 0$};
  \node[region label] at (axis cs:2.85,1.00)
    {\textbf{Quadratic}\\[-1pt]$0<z_i<\delta_i$};
  \node[region label] at (axis cs:5.00,1.20)
    {\textbf{Affine}\\[-1pt]$z_i\geq\delta_i$};

  \addplot[small segment]
    coordinates {(4.8500,0.0000) (4.5733,0.3933)};
  \addplot[small segment]
    coordinates {(4.5733,0.3933) (4.2967,0.7867)};
  \addplot[small segment]
    coordinates {(4.2967,0.7867) (4.0200,1.1800)};
  \addplot[trajectory points]
    coordinates {(4.8500,0.0000) (4.5733,0.3933)
                 (4.2967,0.7867) (4.0200,1.1800)};
  \addplot[selected trajectory]
    coordinates {(4.02,1.18) (2.8805,2.80)};
  \addplot[selected point] coordinates {(2.8805,2.80)};
  \addplot[first point] coordinates {(4.85,0)};

  \nextgroupplot[
    title={(b) Coordinates $i+1$ and $i+2$},
    xlabel={$x_t^{(i+1)}$}
  ]

  \node[anchor=east]
    at (axis description cs:-0.015,0.985) {$x_t^{(i+2)}$};

  \addplot[region fill, fill=inactivefill]
    coordinates {(0,0) (1.18,0) (5.53,4.35) (0,4.35) (0,0)};
  \addplot[region fill, fill=quadraticfill]
    coordinates {(1.18,0) (2.13,0) (5.65,3.52)
                 (5.65,4.35) (5.53,4.35) (1.18,0)};
  \addplot[region fill, fill=affinefill]
    coordinates {(2.13,0) (5.65,0) (5.65,3.52) (2.13,0)};

  \addplot[boundary plot, domain=1.18:5.53] {x-1.18}
    node[pos=0.74, sloped, above, font=\scriptsize, text=boundarygray]
      {$z_{i+1}=0$};
  \addplot[boundary plot, domain=2.13:5.65] {x-2.13}
    node[pos=0.70, sloped, below, font=\scriptsize, text=boundarygray]
      {$z_{i+1}=\delta_{i+1}$};

  \node[region label] at (axis cs:1.30,3.03)
    {\textbf{Inactive}\\[-1pt]$z_{i+1}\leq 0$};
  \node[region label] at (axis cs:2.86,1.42)
    {\textbf{Quadratic}\\[-1pt]$0<z_{i+1}<\delta_{i+1}$};
  \node[region label] at (axis cs:4.55,0.52)
    {\textbf{Affine}\\[-1pt]$z_{i+1}\geq\delta_{i+1}$};

  \addplot[small segment] coordinates {(0.0000,0) (0.3933,0)};
  \addplot[small segment] coordinates {(0.3933,0) (0.7867,0)};
  \addplot[small segment] coordinates {(0.7867,0) (1.1800,0)};
  \addplot[trajectory points]
    coordinates {(0.0000,0) (0.3933,0) (0.7867,0) (1.1800,0)};
  \addplot[selected trajectory]
    coordinates {(1.18,0) (2.80,0)};
  \addplot[selected point] coordinates {(2.80,0)};
  \addplot[first point] coordinates {(0,0)};

\end{groupplot}

\coordinate (legend center) at
  ([yshift=-9.5mm]$(group c1r1.south)!0.5!(group c2r1.south)$);
\matrix[
  matrix of nodes,
  anchor=center,
  column sep=2.8mm,
  row sep=0pt,
  nodes={anchor=base, inner sep=0pt, outer sep=0pt, font=\scriptsize}
] (trajectory-legend) at (legend center) {
  \hspace{3.0mm}$t=t_{i-1}$ &
  \hspace{3.0mm}$t_{i-1}<t<t_i$ &
  \hspace{3.0mm}$t=t_i$ &
  \hspace{9.0mm}updates in the $i$-th gap &
  \hspace{9.0mm}$i$-th checkpoint update \\
};

\fill[black!88, draw=white, line width=0.5pt]
  ([xshift=1.35mm]trajectory-legend-1-1.west) circle (2.2pt);
\fill[smallblue, draw=white, line width=0.5pt]
  ([xshift=1.35mm]trajectory-legend-1-2.west) circle (2.2pt);
\fill[selectedred, draw=white, line width=0.5pt]
  ([xshift=1.35mm]trajectory-legend-1-3.west) circle (2.2pt);
\draw[smallblue, line width=1.45pt,
      -{Stealth[length=1.65mm,width=1.2mm]}]
  ([xshift=0.7mm]trajectory-legend-1-4.west) --
  ([xshift=7.8mm]trajectory-legend-1-4.west);
\draw[selectedred, line width=1.8pt,
      -{Stealth[length=2.0mm,width=1.35mm]}]
  ([xshift=0.7mm]trajectory-legend-1-5.west) --
  ([xshift=7.8mm]trajectory-legend-1-5.west);

\end{tikzpicture}

%% file: sec/003NonAnytime.tex
\section{Stronger Non-Anytime Lower Bound}
\label{sec:nonanytime}

In this section, we sketch the proof of \cref{thm:NonAnytimeLowerBound}, our lower bound on the \textit{non-anytime} GD convergence rate.
We first observe that the lower bound proved in \cref{lem:GeneralLowerBound} holds for any choice of checkpoint timesteps $T=\{t_1<\dotsb<t_k\}\subseteq[N]$ with $k=\abs{T}$.
Hence, to obtain the best (i.e., largest) possible convergence lower bound that applies to all positive step-size schedules $\veta$, we want to \uline{(i) optimize the right-hand side of \mbox{\cref{eq:general_normalized_lower_bound}} in terms of $T\subset[N]$ and (ii) apply the worst-case $\veta$.}
Define
\begin{align}
    \gP(T;\veta):=\prod_{i=1}^{\abs{T}}\frac{b_i(T;\veta)}{2\open{2+s_i(T;\veta)}},\qquad \gP(\emptyset;\veta):=1.
    \label{eq:all_selection_product}
\end{align}
Recall that it equals $A_{k+1}$ defined in the proof of \cref{lem:GeneralLowerBound} with $r_0=1$. Then, \cref{lem:GeneralLowerBound} implies that
\begin{align}
    4 \gR_N(\veta) \ge
    \gB(\veta):=\max_{T\subseteq[N]}\frac{\gP(T;\veta)^2}{1+s_{\abs{T}+1}(T;\veta )}.
    \label{eq:all_selection_objective}
\end{align}
The maximum is attained as it is over the $2^N$ selections.
However, directly maximizing this ratio is difficult.
Our trick is to turn this complicated maximization over all $T\subseteq[N]$ into a \emph{recursively decomposable} scalar problem.
To this end, we define a \textbf{cost} function $\Psi_\lambda$ and its minimum $\gV_\lambda$ in $T$:
\begin{align}
    \Psi_\lambda(T;\veta)
    := \frac{\lambda + s_{\abs{T}+1}(T;\veta)}{\gP(T;\veta)}, \qquad
    \gV_\lambda(\veta)
    = \min_{T\subseteq [N]}  \Psi_\lambda(T;\veta). \label{eq:cost_and_minimum_cost}
\end{align}
Here, we introduce a new parameter $\lambda\ge 2$, which we determine later.
For the empty schedule $\emptyset$ (of length zero), we simply set $\mathcal V_\lambda(\emptyset)=\lambda$.
Then, using the fact that $\frac{1}{1+s} = \max_{\lambda \ge 2} \frac{4(\lambda-1)}{(\lambda + s)^2}$ ($s\ge 0$), we can rewrite the bound in \cref{eq:all_selection_objective} as
\begin{align}
    \gB(\veta) = \max_{\lambda\ge 2} \frac{4(\lambda - 1)}{\gV_\lambda (\veta)^2}. \label{eq:all_selection_objective_reformulated}
\end{align}
See \cref{subsec:TerminalParameterIdentity} for a detailed derivation of it.
Now, it suffices to study the minimum cost $\gV_\lambda(\veta)$. Fortunately, it admits the following exact recursive binary decomposition.
\begin{restatable}[Exact recursive decomposition of $\gV_\lambda(\veta)$]{lemma}{lemExactCostRecursion}
\label{lem:ExactCostRecursion}
    For $\veta_{1:n} = (\eta_t)_{t\in[n]}\in (0,\infty)^n$ and $\lambda\ge2$,
    \begin{align}
        \mathcal V_\lambda(\veta_{1:n})
        =
        \min\left\{
            \lambda+\sum_{t=1}^{n}\eta_t,\;
            \min_{1\le t\le n}\frac{2}{\eta_t}\,
            \mathcal V_2(\veta_{1:t-1})\,
            \mathcal V_\lambda(\veta_{t+1:n})
        \right\}.
        \label{eq:exact_cost_recursion}
    \end{align}
\end{restatable}
\begin{proof}[Proof idea for \cref{lem:ExactCostRecursion}]
    The complete proof appears in \cref{subsec:proof_cost_recursion}.
    The idea is to factorize $\Psi_\lambda(T;\veta_{1:n})$ as
    \begin{align}
        \Psi_\lambda(T;\veta_{1:n})
        =\frac{2}{\eta_t}\cdot
        \Psi_2(T_L;\veta_{1:t-1})
        \cdot
        \Psi_\lambda(T_R;\veta_{t+1:n}),
        \label{eq:cost_factorization_main}
    \end{align}
    by dividing $T\!=\!T_{\mathrm{left}} \cup \{t\} \cup T_{\mathrm{right}}$ for $T_{\mathrm{left}} \!\subseteq [t\!-\!1]$ and $T_{\mathrm{right}} \!\subseteq [n]\setminus[t]$ for $t\!\in\! [n]$, when $T\!\neq\! \emptyset$.
\end{proof}

Next, since we want to obtain the convergence lower bound that applies to all positive schedules $\veta\in (0,\infty)^n$ of any length $n\ge 0$, define the \textbf{worst-case minimum cost} $\gU_n(\lambda)$ over all such $\veta$ by
\begin{align}
    \gU_n(\lambda)
    :=\sup_{\veta\in(0,\infty)^n}\gV_\lambda(\veta),\qquad
    \gU_0(\lambda):=\lambda. \qquad
    (\lambda \ge 2)
    \label{eq:worst_case_cost_definition}
\end{align}
We take a supremum because our lower bound $\gB(\veta)$ is inversely proportional to $\gV_\lambda(\veta)^2$: see \cref{eq:all_selection_objective_reformulated}.
We also define a handy map $\gQ:[2,\infty)^2\to[4,\infty)$ by
\begin{align*}
    \gQ(x,y):=\frac{x+y-2+\sqrt{(x+y-2)^2+8xy}}{2}.
\end{align*}
Equivalently, $w=\gQ(x,y)$ is the unique positive solution of $w(w-x-y+2)=2xy$.
Then, we obtain the following (recursive) bound on $\gU_n(\lambda)$, which is the key technical difficulty of this paper.
\begin{restatable}{lemma}{lemWorstCaseCostBound}
    \label{lem:WorstCaseCostBound}
    For every $n\in \sN$ and $\lambda\ge 2$, the worst-case minimum cost $\gU_n(\lambda)$ (\cref{eq:worst_case_cost_definition}) satisfies
    \begin{align}
        \gU_n(\lambda)
        \le
        \max_{\substack{i,j \ge 0\\i+j = n-1}}
        \gQ \open{\gU_i(2),\gU_j(\lambda)}.
        \label{eq:recursive_worst_case_cost_bound}
    \end{align}
    Furthermore, let $\alpha_\star :=\log_2(1+\sqrt{3})$ and $\nu_\star = 1/\alpha_\star$.
    Then, for all $n\ge 0$, $\lambda\ge 2$,
    \begin{align}
        \gU_n(\lambda)^{\nu_\star}\le \lambda^{\nu_\star}+2^{\nu_\star}n.
        \label{eq:uniform_worst_case_cost_bound}
    \end{align}
\end{restatable}
\begin{proof}[Proof sketch of \cref{lem:WorstCaseCostBound}]
    First, the supremum $\gU_n(\lambda)$ is attained and finite (\cref{lem:ExtremalSchedule}). Choose a maximizer $\veta^\star$ of $\gU_n(\lambda)$ and write $w:=\gV_\lambda(\veta^\star)$. Let
    \[
        \gM:=\{T\subseteq[n]:\Psi_\lambda(T;\veta^\star)=w\}
    \]
    be the family of checkpoint sets $T$ that attains the min-cost particularly for $\veta^\star$.
    Since the cost function $\Psi_\lambda(T;\veta)$ is log-submodular (\cref{lem:CostLogSubmodularity}), $\gM$ is closed under union and intersection, and two nonempty elements of $\gM$ cannot be disjoint.
    Perturbation arguments show that $\emptyset\in\gM$ and that an inclusion-minimal nonempty member must be a singleton $\{\tau\}$.
    Now, set
    \[
        x:=2+\sum_{t<\tau}\eta_t^\star,\qquad
        y:=\lambda+\sum_{t>\tau}\eta_t^\star,\qquad
        \mu:=\eta_\tau^\star.
    \]
    Since both $\emptyset$ and $\{\tau\}$ are both members of $\gM$, by comparing $\Psi_\lambda(\emptyset;\veta^\star)$ and $\Psi_\lambda(\{\tau\};\veta^\star)$, we have
    \[
        w=x+y+\mu-2=\frac{2xy}{\mu},
    \]
    and hence $w=\gQ(x,y)$. 
    Applying \cref{lem:ExactCostRecursion} at $t=\tau$ and the fact that $\gV_\lambda$ is bounded above by the cost of the empty checkpoint set, we have
    \[
        x=\gV_2(\veta_{1:\tau-1}^\star),\qquad
        y=\gV_\lambda(\veta_{\tau+1:n}^\star).
    \]
    Therefore $x\le\gU_{\tau-1}(2)$ and $y\le\gU_{n-\tau}(\lambda)$; monotonicity of $\gQ$ then yields \cref{eq:recursive_worst_case_cost_bound}.

    Note that the map $\gQ$ satisfies $\gQ(x,y)^{\nu_\star} \!<\! x^{\nu_\star} \!+\! y^{\nu_\star}$ for any $x,y\ge 2$ (\cref{lem:QPowerSubadditivity}).
    Then, strong induction over $n\ge 0$ from $\gU_0(\lambda)=\lambda$ gives \cref{eq:uniform_worst_case_cost_bound}:
    \begin{align*}
        \gU_n(\lambda)^{\nu_\star}
        \le \max_{i+j=n-1}\left(\gU_i(2)^{\nu_\star}+\gU_j(\lambda)^{\nu_\star}\right)
        \le \max_{i+j=n-1}\left(2^{\nu_\star}+2^{\nu_\star}i+\lambda^{\nu_\star}+2^{\nu_\star}j\right)
        =\lambda^{\nu_\star}+2^{\nu_\star}n.
    \end{align*}
    See \cref{subsec:proof_worst_case_cost} for the full proof.
\end{proof}

Finally, using the preceding key lemma, we prove our first main theorem, which implies \cref{thm:NonAnytimeLowerBound}.
\begin{restatable}[Non-anytime lower bound]{theorem}{thmNonAnytimeLowerBoundFull}
\label{thm:NonAnytimeLowerBoundFull}
    Fix any $N\in \sN$. For any $N$-step schedule $\veta\in (0,\infty)^N$,
    \begin{align*}
        \mathcal R_N(\veta)\ge\frac{2N^{\alpha_\star}-1}{4(2N)^{2\alpha_\star}} =\Omega(N^{-\alpha_\star}), \qquad \text{where } \alpha_\star = \log_2 (1+\sqrt{3}).
    \end{align*}
\end{restatable}
\begin{proof}[Proof of \cref{thm:NonAnytimeLowerBoundFull}]
    Recall that $\nu_\star = 1/\alpha_\star$.
    By \cref{lem:WorstCaseCostBound}, for any $\lambda\ge2$,
    \begin{align*}
        \mathcal V_\lambda(\veta)\le\mathcal U_N(\lambda)\le\left(\lambda^{\nu_\star}+2^{\nu_\star}N\right)^{\alpha_\star}.
    \end{align*}
    Choose $\lambda_N:=2N^{\alpha_\star}\ (\ge2)$. Since $\lambda_N^{\nu_\star}=2^{\nu_\star}N$ and $\nu_\star \alpha_\star = 1$,
    \begin{align*}
        \mathcal V_{\lambda_N}(\veta)&\le\left(2^{\nu_\star}(2N)\right)^{\alpha_\star}=2(2N)^{\alpha_\star}
    \end{align*}
    Therefore, combining \cref{eq:all_selection_objective,eq:all_selection_objective_reformulated} and applying $2N^{\alpha_\star} - 1 \ge N^{\alpha_\star}$ ($\because N^{\alpha_\star}\ge 1$), we have
    \begin{align*}
        \gR_N(\veta)
        \ge \frac14 \gB(\veta)
        \ge \frac{\lambda_N-1}{\mathcal V_{\lambda_N}(\veta)^2}
        \ge \frac{2N^{\alpha_\star}-1}{4(2N)^{2\alpha_\star}}
        \ge 2^{-(2\alpha_\star+2)}N^{-\alpha_\star}.
    \end{align*}
\end{proof}

\paragraph{Why $\alpha_\star=\log_2(1+\sqrt{3})$?}
A nontrivial part of the proof is where the exponent $\alpha_\star$ appears for the first time.
It is relevant to a property of the map $\gQ$ (\cref{lem:QPowerSubadditivity}): $\gQ(x,y)^{\nu_\star} \!<\! x^{\nu_\star} \!+\! y^{\nu_\star}$ for any $x,y\!\ge\! 2$ and  $\nu_\star=1/\alpha_\star$.
The proof (sketch) of \cref{lem:WorstCaseCostBound} exploits this fact to convert a complicated recursion in \cref{eq:recursive_worst_case_cost_bound} into a subadditivity after taking power $\nu_\star$, i.e., $\gU_n (\lambda)^{\nu_\star} \le \gU_i(2)^{\nu_\star}+\gU_j(\lambda)^{\nu_\star}$ for some $i,j \ge 0$ such that $i+j = n-1$.
Now, let us provide a hand-wavy derivation of the exponent: we want to obtain the largest exponent $\nu$ satisfying $\gQ(x,y)^{\nu} \!\le\! x^{\nu} \!+\! y^{\nu}$.
When $x=y$ and for large enough $x$, observe that $\gQ(x,x) = x-1 + \sqrt{3x^2 - 2x + 1} \le (1+\sqrt{3}) x$, and we hope for an inequality $\gQ(x,x)^\nu \le 2x^\nu$.
Hence, we need $(1+\sqrt{3})^\nu \le 2$.
The largest possible exponent is thus $\nu_\star = \log_{(1+\sqrt{3})} 2 = 1/\alpha_\star$.
The rigorous proof of \cref{lem:QPowerSubadditivity} can be found in the appendix.

%% file: sec/004Anytime.tex
\section{Stronger Anytime Lower Bound}
\label{sec:anytime}

So far, we have considered an arbitrarily fixed number of iterations $N$; we could take the worst-case schedule $\veta$ for each $N$.
We now turn to the \textit{anytime} setup, and sketch the proof of \cref{thm:AnytimeLowerBound}.
Let us first fix an infinite positive schedule $\veta=(\eta_t)_{t\ge 1} \in (0,\infty)^\sN$ and study its every finite prefix $\veta_{1:n}$.
So the prefixes are coupled across the horizon $n\in \sN$.

Define three key quantities: the running sum $S_n$, the running maximum $M_n$, and the rate $r_n$:
\begin{align*}
    S_n:=\sum_{t=1}^n\eta_t, \qquad
    M_n:=\max_{1\le t\le n} \eta_t, \qquad
    r_n:=\mathcal R_n(\veta_{1:n}).
\end{align*}
We call $n$ a \emph{record time} if $\eta_n = M_n$: a step where it is the largest seen so far.
We aim to obtain inequalities between the key quantities and then eliminate $S_n$ and $M_n$ to yield a bound in terms of $r_n$ for large enough record times $n\in \sN$.

We start with two simple lower bounds of $r_n$.
First, if we take an empty checkpoint set $T=\emptyset$ for $\veta_{1:n}$, \cref{lem:GeneralLowerBound} gives
\begin{align}
    r_n \ge \frac{1}{4(1+S_n)}. \label{eq:naive_bound_empty_checkpoint}
\end{align}
Second, for a record time $n$, taking only the last step of $\veta_{1:n}$ as a single checkpoint (i.e., $T=\{n\}$),  \cref{lem:GeneralLowerBound} yields
\begin{align}
    r_n \ge \frac{1}{4}\open{\frac{\eta_n}{2(2+S_{n-1})}}^2 = \frac{1}{16}\open{\frac{M_n}{2+S_{n-1}}}^2. \label{eq:naive_bound_last_checkpoint_record_time}
\end{align}

The key lemma for this section is as follows.
It provides an upper bound for $S_n$ in terms of $M_n$, given that $r_n<1/4$.
Recall that $\nu_\star = 1/\alpha_\star = 1/\log_2(1+\sqrt{3})$.
\begin{restatable}[Sum bound at record times]{lemma}
{lemImprovedRecordSumBound}
\label{lem:ImprovedRecordSumBound}
    Suppose $n\ge 2$ is a record time of $\veta \in (0,\infty)^{\sN}$ and assume that $r_n\le 1/4$. Then,
    \begin{align*}
        S_n\le C\,nM_n^{1-\nu_\star}, \qquad 
        \text{where } C:=\frac{4^{\nu_\star}}{1-\nu_\star}.
    \end{align*}
\end{restatable}
\begin{proof}[Proof Sketch of \cref{lem:ImprovedRecordSumBound}]
    Combining \cref{lem:WorstCaseCostBound} and the factorization in \cref{eq:cost_factorization_main}, we obtain a general sum-maximum bound (\cref{lem:OrderSensitiveSumMaximum}), roughly in order of $S_n + M_n \lesssim (n+1) M^{1-\nu_\star}$.
    Then, applying \cref{lem:GeneralLowerBound} for a checkpoint set $T_+ := T\cup \{n\}$ with $T\subseteq[n-1]$, to yield a cleaner bound $S_n \lesssim n M_n^{1-\nu_\star}$ for a record time $n$.
    See \cref{sec:anytime_proofs} for the full proof.
\end{proof}

Combining the three inequalities, \cref{eq:naive_bound_empty_checkpoint,eq:naive_bound_last_checkpoint_record_time,lem:ImprovedRecordSumBound}, we now eventually prove our second main theorem, which implies \cref{thm:AnytimeLowerBound}.
\begin{restatable}[Anytime lower bound]{theorem}{thmAnytimeLowerBoundFull}
\label{thm:AnytimeLowerBoundFull}
    There exists a universal constant $C_{\star}>0$ such that every positive infinite step-size schedule $\veta=(\eta_t)_{t\ge 1}$ satisfies
    \begin{align}
        \limsup_{n\to\infty}n^{\beta_\star}\gR_n(\veta_{1:n})\ge C_{\star}, \qquad \text{where } \beta_\star = \frac{2\alpha_\star}{1+\alpha_\star} = \frac{2\log_2 (1+\sqrt{3})}{1+\log_2 (1+\sqrt{3})}.
        \label{eq:improved_anytime_limsup}
    \end{align}
    In particular, no positive infinite schedule satisfies
    \begin{align}
        \gR_n(\veta_{1:n})=o(n^{-\beta_\star}).
        \label{eq:improved_anytime_no_little_o}
    \end{align}
\end{restatable}
\begin{proof}[Proof of \cref{thm:AnytimeLowerBoundFull}]
    We distinguish cases based on whether the step-sizes in $\veta$ are bounded.
    
    \paragraph{Case 1: bounded step-sizes.}
    We have $\sup_n M_n=:M_\infty<\infty$; $M_\infty$ is a uniform bound of $\veta$.
    Since $S_n \le n\, M_n \le n\, M_\infty$, a naive bound \cref{eq:naive_bound_empty_checkpoint} gives
    \begin{align*}
        r_n\ge\frac{1}{4(1+nM_\infty)}, \qquad n\in \sN.
    \end{align*}
    Since $\beta_\star>1$, it holds that $n^{\beta_\star}r_n\to \infty$.
    Thus, \cref{eq:improved_anytime_limsup} holds for any choice of $C_{\star}>0$.

    \paragraph{Case 2: unbounded step-sizes.}
    Then $M_n\to\infty$; thus, there are infinitely many record times.
    If $r_n$ does not tend to $0$ along the record times, there exists $\varepsilon>0$ and infinitely many record times with $r_n\ge \varepsilon$, along which $n^{\beta_\star} r_n\to\infty$.
    Hence, we can safely assume $r_n\to 0$ along the record times $n$.

    Fix a record time $n\ge 2$ large enough that $r_n < 1/4$. 
    Then, \cref{eq:naive_bound_empty_checkpoint} gives
    \begin{align}
        S_n\ge \frac{1}{4r_n}-1=\frac{1-4r_n}{4r_n} >0
        \label{eq:Sn_bound_from_naive}
    \end{align}
    and \cref{eq:naive_bound_last_checkpoint_record_time} gives
    \begin{align*}
        M_n \le 4 (2+S_{n-1}) \sqrt{r_n} = 4 (2+S_n-M_n) \sqrt{r_n}.
    \end{align*}
    where we used $S_n-S_{n-1}=\eta_n =M_n$.
    Solving for $M_n$ yields
    \begin{align}
        M_n \le \frac{4 \sqrt{r_n}}{1+4 \sqrt{r_n}} (2+S_{n}).
        \label{eq:Mn_bound_from_naive}
    \end{align}
    Substitute \cref{eq:Mn_bound_from_naive} into  \cref{lem:ImprovedRecordSumBound} and divide both sides by $S_n$.
    Then, writing $\nu = \nu_\star$ and $\bar\nu = 1-\nu_\star$ for brevity, we have
    \begin{align*}
        1\le C n \cdot \open{\frac{4 \sqrt{r_n}}{1+4 \sqrt{r_n}}}^{\bar\nu} \open{1+\frac{2}{S_n}}^{\bar\nu} S_n^{-\nu}
    \end{align*}
    Applying \cref{eq:Sn_bound_from_naive},
    \begin{align*}
        1&\le C n \cdot \open{\frac{4 \sqrt{r_n}}{1+4 \sqrt{r_n}}}^{\bar\nu} \open{\frac{1+4r_n}{1-4r_n}}^{\bar\nu} \open{\frac{4r_n}{1-4r_n}}^{\nu} \\
        &=\frac{4 C n r_n^{\nu+(\bar\nu/2)}}{1-4r_n} \open{\frac{1+4r_n}{1+4 \sqrt{r_n}}}^{\bar\nu}.
    \end{align*}
    Hence, since $\nu + \frac{\bar\nu}{2} = \frac{1+\nu}{2} = \frac{1+\alpha_\star}{2\alpha_\star} = \frac{1}{\beta_\star}$,
    \begin{align*}
        n\, r_n^{1/\beta_\star} \ge \frac{1-4r_n}{4 C }\open{\frac{1+4 \sqrt{r_n}}{1+4r_n}}^{\bar\nu}.
    \end{align*}
    Recall that we assume $r_n \to 0$ along the record times.
    Taking exponent $\beta_\star$ and $\limsup$ to both sides, 
    \begin{align*}
        \limsup_{n\to\infty} n^{\beta_\star} r_n \ge (4C)^{-\beta_\star} = \frac{(1-\nu_\star)^{\beta_\star}}{16},
    \end{align*}
    which proves \cref{eq:improved_anytime_limsup} for $C_\star := \frac{(1-\nu_\star)^{\beta_\star}}{16} > 0$.
    The statement \cref{eq:improved_anytime_no_little_o} follows immediately.
\end{proof}

\paragraph{Why $\beta_\star = \frac{2\alpha_\star}{1+\alpha_\star}$?}
To provide a bird's-eye view of the proof above, let us explain the crux of it.
We combine the following three bounds:
\begin{align*}
    S_n \gtrsim r_n^{-1}, \qquad
    M_n \lesssim S_{n-1} \sqrt{r_n}, \qquad
    S_n \lesssim n M_n^{1-\nu_\star}.
\end{align*}
In particular, the last bound is from \cref{lem:ImprovedRecordSumBound}.
We also apply the fact that $n$ is a record time: as a result, $S_{n-1} = S_n - M_n$, which implies $M_n \lesssim S_n r_n^{1/2}$.
By briefly combining these, we obtain
\begin{align*}
    1 \lesssim n M_n^{1-\nu_\star}S_n^{-1} \lesssim n r_n^{\frac{1-\nu_\star}{2}} S_n^{-\nu_\star}\lesssim n r_n^{\frac{1+\nu_\star}2}.
\end{align*}
Thus, since $\beta_\star = \frac{2\alpha_\star}{1+\alpha_\star}=\frac{2}{1+\nu_\star}$, it follows that $n^{\beta_\star} r_n \gtrsim 1$ for infinitely many record times $n$.

%% file: sec/009Conclusion.tex
\section{Discussion}

We provided two novel, improved lower bounds on GD convergence rates for unconstrained smooth convex minimization.
Our first main result (\cref{thm:NonAnytimeLowerBound,thm:NonAnytimeLowerBoundFull}) was for the \textit{non-anytime} setup, where we could (conceptually) determine the worst-case schedule depending on every time horizon $N\ge 1$.
In this setup, we showed a lower bound $\Omega(N^{-\alpha_\star})$ in \cref{sec:nonanytime}, where $\alpha_\star=\log_2(1+\sqrt{3})\approx 1.450$.
On the other hand, our second result (\cref{thm:AnytimeLowerBound,thm:AnytimeLowerBoundFull}) was for the \textit{anytime} setup, where we should analyze the convergence rate of a given infinite schedule that applies to an arbitrary choice of time horizon.
In this case, we showed that \textit{no} infinte step-size schedule can attain $o(n^{-\beta_\star})$ rate in \cref{sec:anytime}, where $\beta_\star = \frac{2\alpha_\star}{1+\alpha_\star} \approx 1.184$.

\paragraph{The exponents $\alpha_\star$ \& $\beta_\star$.}
Our results show a particularly interesting similarity between the proven lower-bound exponents and the best-known upper-bound exponents.
Namely, in the non-anytime setup, the best-known upper-bound exponent is $\alpha_{\color{darkgreen}\mathrm{sil}}:= \log_2 (1+\sqrt{\color{darkgreen}2})$~\citep{altschuler2025acceleration,grimmer2023accelerated,zhang2026accelerated}, while our lower-bound exponent is $\alpha_{\color{red}\star} = \log_2(1+\sqrt{\color{red}3})$.
In the anytime setup, \citet{zhang2025anytime} show the $\frac{2\alpha_{\color{darkgreen}\mathrm{sil}}}{1+\alpha_{\color{darkgreen}\mathrm{sil}}}$ exponent upper bound, whereas we showed the $\beta_{\color{red}\star} =\frac{2\alpha_{\color{red}\star}}{1+\alpha_{\color{red}\star}}$ exponent lower bound.

\paragraph{Hard instance function $F$.}
In \cref{sec:hard_instance}, we defined our hard instance function by the sum of one-sided Huber component functions.
Each component function $H_{\delta_i}$ depends on two adjacent coordinates of the input, $x^{(i)}$ and $x^{(i+1)}$; each pair of consecutive component functions shares a single dependent variable.
Indeed, variants of Huber functions have been applied to compute convergence lower bounds in convex smooth settings in prior works~\citep{kornowski2024open,tsai2026lower,ma2026lower}.
In particular, the proof of Theorem 2 by \citet{kornowski2024open} applies a two-sided Huber function by showing that an overshoot may happen if the last step-size is large.
To the best of our knowledge, our construction of $F$ is novel: we create a chain of overshoots at (possibly large) checkpoint steps by activating at most a single component at once (see \cref{fig:one_block_dynamics}).

\paragraph{Limitations.}
Still, this work has not completely closed the gap between upper and lower convergence bounds, although it has substantially narrowed the gap.
We strongly believe our proofs of (non-)anytime lower bounds in \cref{sec:nonanytime,sec:anytime} are quite tight.
Nevertheless, we suspect the gap is due to our construction of a hard instance in \cref{sec:hard_instance}; we believe that the convergence rate lower bounds can be improved by a slightly different construction of the hard instance function if it results in a better (larger) checkpoint-dependent lower bound than our \cref{lem:GeneralLowerBound}.

%% file: sec/099Ethics.tex
\subsection*{AI use statement}

In this work, we used generative AI tools (e.g., ChatGPT 5.6 sol) to assist with writing proofs, proposing or refining hypotheses, and translation.
We have not used generative AI tools to develop theoretical models or conceptual frameworks, formulate mathematical claims, design or provide feedback on research methodology or experiments, or interpret results; generating synthetic data sets, implementing methods, cleaning and reformatting datasets, and supporting qualitative and thematic data analysis are not applicable to this work.
Additionally, we used generative AI tools to draft parts of the research paper, summarize or analyze existing literature, brainstorm, source/search for information, edit the paper to improve readability, identify relevant literature, and format references.
We have reviewed all AI-assisted work:
the correctness of the polished texts and refined proofs generated by LLM were verified line by line and totally reorganized by all three authors.
We take responsibility for the final content of this work, including text, claims, or artifacts produced with the aid of generative AI.

%% file: sec/900TableOfContents.tex
\part*{\large Table of Contents}

\begingroup
\hypersetup{linkcolor = black}  
\etocsettocstyle{}{}  
\etocsetnexttocdepth{subsection} 
\tableofcontents
\endgroup

%% file: sec/901Rescaling.tex
\section{Justification of Rescaling}
\label{sec:rescaling}

In \cref{sec:intro}, we define a convergence rate $\gR^{L}_N(\veta)$ for a general smoothness parameter $L>0$ as \cref{eq:convergence_rate_R_N}.
In this section, we show that it is enough to study the case of $L=1$, thereby providing a justification for studying $\gR_N(\veta) = \gR^{1}_N(\veta)$ in the rest of this paper.

\begin{restatable}[Rescaling preserves the trajectory]{lemma}{lemRescalingTrajectory}
\label{lem:RescalingTrajectory}
Consider a positive step-size schedule $\veta=(\eta_1,\ldots, \eta_N)\in(0,\infty)^N$ and a function $f:\R^d\rightarrow \R$ in $ \Fclass_L(\R^d)$. Define the rescaled schedule $\tilde\veta=(\tilde\eta_1,\ldots,\tilde\eta_N)$ and the rescaled function $\tilde f:\R^d\rightarrow \R$ by
\begin{align*}
\tilde\eta_t=L\eta_t, \qquad \tilde f(\vx)=\frac{1}{L}f(\vx).
\end{align*}
Then $\tilde f\in \Fclass_1(\R^d)$ and $\argmin \tilde f = \argmin f$.

Moreover, starting from the same initial point $\vx_0=\tilde\vx_0$, let $\vx_1,\cdots,\vx_N$ and $\tilde\vx_1,\cdots,\tilde\vx_N$ be the GD iterates defined by
\begin{align*}
    \vx_t = \vx_{t-1} - \eta_t \nabla f (\vx_{t-1}), \qquad \tilde\vx_t = \tilde\vx_{t-1} - \tilde\eta_t \nabla \tilde f (\tilde\vx_{t-1}),\qquad t \in [N].
\end{align*}
Then
\begin{align}
    \vx_t = \tilde\vx_t \qquad \text{for all }t=0,\ldots,N.
    \label{eq:rescaling_same_trajectory}
\end{align}
\end{restatable}
\begin{proof}[Proof of \cref{lem:RescalingTrajectory}]
Since $f$ is convex, $\tilde f = \frac{1}{L}f$ is convex. Moreover, for every $\vx,\vy\in \R^d$,
\begin{align*}
    \norm{\nabla\tilde f(\vx)-\nabla\tilde f(\vy)}&=\frac{1}{L}\norm{\nabla f(\vx)-\nabla f(\vy)}\nonumber\\
    &\le \norm{\vx-\vy}.
\end{align*}
Since multiplication by a positive constant does not change the minimizer set, $\argmin \tilde f = \argmin f$. Hence $\tilde f$ is in $\Fclass_1(\R^d)$. 

We now compare the two GD trajectories. Suppose that $\tilde \vx_{t-1}=\vx_{t-1}$. Then
\begin{align*}
    \tilde\vx_{t}&=\tilde\vx_{t-1}-\tilde\eta_t\nabla \tilde f(\tilde\vx_{t-1})\nonumber\\
    &=\vx_{t-1}- \frac{\tilde\eta_t}{L}\nabla f(\vx_{t-1})\nonumber\\
    &=\vx_{t-1}-\eta_t\nabla f(\vx_{t-1})=\vx_t\nonumber.
\end{align*}
Since $\tilde\vx_0=\vx_0$, induction on $t$ proves \cref{eq:rescaling_same_trajectory}.
\end{proof}
Using the preceding lemma, we can derive $\gR_N^L(\veta)=\gR_N^1(\tilde \veta)$:
\begin{align*}
    \gR_N^L(\veta)&= \sup_{\substack{d\in \sN,\ f \in \Fclass_{L}(\R^d),\\
    \vx^\star \in X^\star_{f},\ \vx_{0} \in \R^d\setminus X^\star_{f}} }\frac{f(\vx_{N}) - f(\vx^\star)}{L \norm{\vx_{0} - \vx^\star}^2}\nonumber\\
    &=\sup_{\substack{d\in \sN,\ f \in \Fclass_{L}(\R^d),\\
    \vx^\star \in X^\star_{f},\ \vx_{0} \in \R^d\setminus X^\star_{f}}} \frac{\tilde f(\tilde \vx_{N}) - \tilde f(\vx^\star)}{\norm{\vx_{0} - \vx^\star}^2}\nonumber\quad \text{\color{lightgray}(since $ \tilde f=\tfrac{1}{L}f$ and $\tilde \vx_N=\vx_N$)}\\
    &=\sup_{\substack{d\in \sN,\ \tilde f \in \Fclass_{1}(\R^d),\\
    \vx^\star \in X^\star_{\tilde f},\ \vx_{0} \in \R^d\setminus X^\star_{\tilde f}}} \frac{\tilde f(\tilde \vx_{N}) - \tilde f(\vx^\star)}{\norm{\vx_{0} - \vx^\star}^2}\nonumber\quad \text{\color{lightgray}(since $ f \mapsto \tilde f$ is a bijection and $X^\star_{f}=X^\star_{\tilde f}$)}\\
    &=\gR_N^1(\tilde\veta)
\end{align*}

Finally, for a fixed length $N\ge 1$, the map $\veta \longmapsto \tilde\veta$ is a bijection from $(0,\infty)^N$ to itself. Therefore, considering arbitrary positive step-size
schedules for general $L>0$ is equivalent to considering arbitrary
positive step-size schedules in the normalized case $L=1$.
Hence, without loss of generality, we set $L=1$ throughout the rest of
the paper.

%% file: sec/902HardInstanceProofs.tex
\section{Proofs: Hard Instance}

\subsection{\texorpdfstring{Proof of \cref{prop:HardInstanceProperties}}{Proof of Prop 2.1}}
\label{subsec:hard_instance_properties}
We restate the proposition for readability.
\propHardInstanceProperties*
\begin{proof}[Proof of \cref{prop:HardInstanceProperties}]
    Recall that one-sided Huber function $H_{\delta}$ (\cref{eq:one_sided_huber}) is nonnegative, convex, and 1-smooth.
    The nonnegativity of $F$ follows directly. Also, $F$ is convex since a composition of a convex function and an affine function is convex.
    
    To show Lipschitz smoothness, let us fix $\vx, \vy \in \R^{k+1}$. Since $H_{\delta_i}'$ is nondecreasing and 1-Lipschitz, we can take $\vartheta_i \in [0,1]$ for each $i\in [k+1]$ such that
    \begin{align*}
        H_{\delta_i}'(z_i(\vx)) - H_{\delta_i}'(z_i(\vy)) 
        = \vartheta_i \cdot \open{z_i(\vx) - z_i(\vy)} 
        = \begin{dcases}
            \vartheta_i \cdot (\vx - \vy)^\top (\ve_i - \ve_{i+1}) & \text{if } i \in [k],\\
            \vartheta_{k+1} \cdot (\vx - \vy)^\top \ve_{k+1} & \text{if } i = k+1.
        \end{dcases}
    \end{align*}
    Thus, we have $\nabla F(\vx) - \nabla F(\vy) = \mM (\vx - \vy)$ for a matrix $\mM \in \R^{(k+1)\times(k+1)}$ defined as
    \begin{align*}
        \mM := \sum_{i=1}^k \frac{\vartheta_i}{4} \cdot (\ve_i - \ve_{i+1})(\ve_i - \ve_{i+1})^\top + \frac{\vartheta_{k+1}}{2} \cdot \ve_{k+1} \ve_{k+1}^\top.
    \end{align*}
    Take any $\vv = \open{v^{(1)}, \cdots, v^{(k+1)}} \in \R^{k+1}$. Observe that
    \begin{align*}
        \vv^\top \mM \vv &= \sum_{i=1}^k \frac{\vartheta_i}{4} \cdot \open{v^{(i)} - v^{(i+1)}}^2 + \frac{\vartheta_{k+1}}{2} \cdot \open{v^{(k+1)}}^2 \\
        &\le \sum_{i=1}^k \frac12 \closed{ \open{v^{(i)}}^2 + \open{v^{(i+1)}}^2 } + \frac12 \open{v^{(k+1)}}^2\\
        &\le \sum_{i=1}^{k+1} \open{v^{(i)}}^2 \\
        &= \norm{\vv}_2^2.
    \end{align*}
    The inequality in the second line above holds because $\vartheta_i \le 1$ and $\open{a-b}^2 \le 2\open{a^2+b^2}$ for any $a,b\in \R$.
    This proves $\norm{\mM}_2 \le 1$ ($\because \mM \preceq \mI$).
    Using this, we now have that $F$ is 1-smooth, as 
    \begin{align*}
        \norm{\nabla F(\vx) - \nabla F(\vy)}_2 \le \norm{\mM}_2 \norm{\vx - \vy}_2 \le  \norm{\vx - \vy}_2.
    \end{align*}

    Lastly, $z_i(\vzero) = -c_i \le 0$ and $H_{\delta_i} (-c_i) = 0$ for all $i\in [k+1]$, so $F(\vzero) = 0$. Since $F$ is nonnegative, $\vzero$ is indeed a minimizer of $F$.
\end{proof}

\subsection{\texorpdfstring{Proof of \cref{lem:OneStepGD}}{Proof of lem 2.2}}
We restate the lemma for readability.
\label{subsec:one_step_dynamics}
\lemOneStepGD*
\begin{proof}[Proof of \cref{lem:OneStepGD}]
    Define $g_i(\cdot):= H'_{\delta_i}(z_i(\cdot))= \min\{\max\{0, z_i(\cdot)\}, \delta_i\}$.
    In particular, it holds that:
    \begin{align*}
        \begin{dcases}
            g_i(\vx) = 0 & \text{when }z_i(\vx) \le 0, \text{ i.e., $\vx$ is inactive at component $i$;} \\
            g_i(\vx) = \delta_i & \text{when }z_i(\vx) \ge \delta_i, \text{ i.e., $\vx$ is in its $i$-th nonzero affine regime.}
        \end{dcases}
    \end{align*}

    If $k=0$, the lemma follows immediately. Thus we only consider the case $k\ge 1$. 
    
    For $k\ge 2$, the margins at $\vx_+$ satisfy
    \begin{align}
    \begin{dcases}
        z_1(\vx_+) &= z_1(\vx) \phantom{+ \frac{\eta}{4} g_{i-1}(\vx),} - \frac{\eta}{2} g_1(\vx) + \frac{\eta}{4} g_2(\vx), \\
        z_i(\vx_+) &= z_i(\vx) + \frac{\eta}{4} g_{i-1}(\vx) - \frac{\eta}{2} g_{i}(\vx) + \frac{\eta}{4} g_{i+1}(\vx), \qquad (1<i<k) \\
        z_k(\vx_+) &= z_k(\vx) + \frac{\eta}{4} g_{k-1}(\vx) - \frac{\eta}{2} g_{k}(\vx) + \frac{\eta}{2} g_{k+1}(\vx), \\
        z_{k+1}(\vx^+) &= z_{k+1}(\vx) + \frac{\eta}{4} g_{k}(\vx) - \frac{\eta}{2} g_{k+1}(\vx).
    \end{dcases} \label{eq:margin_one_step_update}
    \end{align}
    For $k=1$, the margins satisfy
    \begin{align}
    \begin{dcases}
        z_1(\vx_+) &= z_1(\vx)  - \frac{\eta}{2} g_1(\vx) + \frac{\eta}{2} g_2(\vx), \\
        z_{2}(\vx_+) &= z_{2}(\vx) + \frac{\eta}{4} g_{1}(\vx) - \frac{\eta}{2} g_{2}(\vx).
    \end{dcases} \label{eq:margin_one_step_update_k1}
    \end{align}
    These identities will be used below to prove \cref{item:margin_one_step_lower_bound}.
    Observe that
    \begin{align*}
        z_i(\vx) = \begin{dcases}
            \vx^\top (\ve_i - \ve_{i+1})-c_i & \text{if } i\in [k]; \\
            \vx^\top \ve_{k+1}-c_{k+1} & \text{if } i = k+1.
        \end{dcases}
    \end{align*}
    We can use this to obtain
    \begin{align}
    \begin{aligned}
        \nabla F(\vx) 
        &= \frac14 \sum_{i=1}^k H'_{\delta_i} (z_i(\vx)) (\ve_i - \ve_{i+1}) + \frac12 H'_{\delta_{k+1}} (z_{k+1}(\vx)) \ve_{k+1} \\
        &= \frac14 \sum_{i=1}^k g_i(\vx) (\ve_i - \ve_{i+1}) + \frac12 g_{k+1}(\vx) \ve_{k+1}
    \end{aligned} \label{eq:grad_F}
    \end{align}
    Thus, \cref{eq:margin_one_step_update} and \cref{eq:margin_one_step_update_k1} can be easily derived from the fact
    \begin{align*}
        z_i(\vx_+) - z_i(\vx) 
        = \begin{dcases}
            -\eta \nabla F(\vx)^\top (\ve_i - \ve_{i+1}) & \text{if } i\in [k]; \\
            -\eta \nabla F(\vx)^\top \ve_{k+1} & \text{if } i = k+1.
        \end{dcases}
    \end{align*}
    
    Consequently, since $z_i(\vx) \ge \delta_i$ implies $g_i(\vx) = \delta_i$, \cref{item:margin_one_step_lower_bound} follows from $g_i(\vx) \ge 0~(\forall i)$.
    
    Moreover, under the assumption of \cref{item:margin_one_step_edge}, we have $g_i(\vx) = \delta_i$ and $g_{i+1}(\vx) = \cdots = g_{k+1}(\vx) = 0$. Thus, from \cref{eq:grad_F},
    \begin{align*}
        x_+^{(i+1)} - x^{(i+1)} &= -\eta\nabla F(\vx)^\top \ve_{i+1} = \frac{\eta}{4} g_i(\vx) - 0 = \frac{\eta \delta_i}{4}, \\
        x_+^{(j')} - x^{(j')} &= -\eta\nabla F(\vx)^\top \ve_{j'} = 0 \quad (i+2 \le j' \le k+1).
    \end{align*}
    This proves \cref{item:margin_one_step_edge}, concluding the proof.
\end{proof}

\vspace{-5pt}
\subsection{Lemmas for General Lower Bound}
\vspace{-5pt}

We restate the definition of the parameters for readability.

With $c_1=0$ and $A_1 = r_0$, we set
\begin{align}
    \delta_i = \frac{2A_i}{2 + s_i},~~
    c_{i+1} = \frac{s_i \delta_i}{4},~~
    A_{i+1} = \frac{b_i \delta_i}{4},~~ (\forall i\in[k]); \qquad  
    \delta_{k+1} = \frac{A_{k+1}}{1 + s_{k+1}}.  
    \label{eq:param_recursive_again}
\end{align}

Rolling out the recursion above, the parameters can be explicitly expressed by $r_0>0$ (initial position), $(s_i)_{i=1}^{k+1}$ (small step partial sums), and $(b_i)_{i=1}^{k}$ (big steps):
\begingroup
\allowdisplaybreaks
\begin{align}
    A_i &= r_0 \prod_{j=1}^{i-1} \frac{b_j}{2(s_j + 2)}, &&i \in [k+1]; \label{eq:parameter_A_i_rolledout}
    \\
    c_{i+1} &= \frac{r_0 s_i}{2(s_i + 2)} \prod_{j=1}^{i-1} \frac{b_j}{2(s_j + 2)}, && i \in [k];
    \nonumber\\
    \delta_i &= \frac{2 r_0}{s_i + 2} \prod_{j=1}^{i-1} \frac{b_j}{2(s_j + 2)}, && i \in [k];
    \nonumber\\
    \delta_{k+1} &= \frac{r_0}{s_{k+1} + 1} \prod_{j=1}^{k} \frac{b_j}{2(s_j + 2)}. \nonumber
\end{align}
\endgroup


    

\begin{restatable}[Gaps stay in current affine region]{lemma}{lemGapsStayAffine}
\label{lem:GapsStayAffine}
    Fix any $i \in [k]$.
    Suppose that, at the beginning of a block $i$ (i.e., $t=t_{i-1}$), 
    $z_i(\vx_{t_{i-1}}) = A_i$ and $x_{t_{i-1}}^{(j)} = 0$ for $j \ge i+1$.
    Then, for all $t=t_{i-1},\ldots,t_i-1$, if we write $\sigma_{i,t} := \sum_{\tau = t_{i-1} + 1}^t \eta_\tau$, the GD iterate $\vx_t$ satisfies:
    \begin{align}
        z_i(\vx_t) \ge A_i - \frac{\sigma_{i,t} \delta_i}{2} \ge \delta_i, \quad
        x_t^{(i+1)} = \frac{\sigma_{i,t} \delta_i}{4}, \quad
        \text{and } x_t^{(j)} = 0 \quad (j \ge i+2). \label{eq:small_step_dynamics}
    \end{align}
    As a result, the GD iterate $\vx_t$ is in its $i$-th nonzero affine regime (i.e., $z_i(\vx_t) \ge \delta_i$) and inactive at every component $j\ge i+1$ (i.e., $z_j(\vx_t) \le 0$).
\end{restatable}
\begin{proof}[Proof of \cref{lem:GapsStayAffine}]
    We proceed with induction in $t=t_{i-1},\ldots,t_i-1$ to prove \cref{eq:small_step_dynamics}. 
    First, due to \cref{eq:param_recursive_again}, we have $A_i = \frac{s_i  +2}{2}\cdot \delta_i \ge \delta_i$ and $c_{i+1} = \frac{s_i\delta_i}{4}$.
    Thus, since $\sigma_{i,t_{i-1}}=0$, \cref{eq:small_step_dynamics} is true at $t=t_{i-1}$ (base case).
    
    Next, take any $t$ such that $t_{i-1} + 1 \le t \le t_i - 1$ and impose 
    the inductive hypothesis (namely, assuming \cref{eq:small_step_dynamics} at step $t-1$).
    Then, we can apply \cref{lem:OneStepGD}'s~\cref{item:margin_one_step_lower_bound,item:margin_one_step_edge}, as well as the facts
    \begin{align*}
        \sigma_{i,t-1} + \eta_t = \sigma_{i,t} \quad \text{and} \quad \sigma_{i,t} \le s_i
    \end{align*}
    to deduce the following:
    \begin{align*}
        z_{i}(\vx_t)
        &\stackrel{\text{\cref{lem:OneStepGD}~\ref{item:margin_one_step_lower_bound}}}{\ge} z_{i}(\vx_{t-1}) - \frac{\eta_t \delta_{i}}{2}
        \stackrel{\text{\cref{eq:small_step_dynamics}}}{\ge} A_i - \frac{\sigma_{i,t} \delta_{i}}{2} \ge A_i - \frac{s_i \delta_i}{2} = \delta_i; \\
        x_{t}^{(i+1)} 
        &\stackrel{\text{\cref{lem:OneStepGD}~\ref{item:margin_one_step_edge}}}{=} x_{t-1}^{(i+1)} + \frac{\eta_t \delta_i}{4}
        \stackrel{\text{\cref{eq:small_step_dynamics}}}{=} \frac{\sigma_{i,t} \delta_i}{4}; \\
        x_{t}^{(j)} 
        &\stackrel{\text{\cref{lem:OneStepGD}~\ref{item:margin_one_step_edge}}}{=} x_{t-1}^{(j)}
        \stackrel{\text{\cref{eq:small_step_dynamics}}}{=} 0 \quad (j \ge i+2).
    \end{align*}
    Hence, it proves \cref{eq:small_step_dynamics} completely. As a result, the first line above shows that $z_i(\vx_t) \ge \delta_i$ for all $t=t_{i-1},\ldots, t_i - 1$.
    Moreover, the second and third lines imply $z_j (\vx_t) \le 0$ for $t=t_{i-1},\ldots, t_i - 1$ and $j=i+1,\ldots,k+1$, because
    \begin{align*}
        z_{i+1}(\vx_t)
        &= x_{t}^{(i+1)} - x_{t}^{(i+2)} - c_{i+1}
        = \frac{\delta_i}{4} (\sigma_{i,t} - s_i) \le 0;
        \\
        z_{j}(\vx_t)
        &= x_{t}^{(j)} - x_{t}^{(j+1)} - c_{j}
        = -c_j \le 0 \quad (j \ge i+2).
    \end{align*}
\end{proof}

\begin{restatable}[Big step activates the next block]{lemma}{lemBigStepActivateNext}
\label{lem:BigStepActivatesNext}
    Fix any $i\in[k]$.
    Suppose the same assumption as \cref{lem:GapsStayAffine}.
    Then, after the big step $b_i=\eta_{t_i}$, we have
    \begin{align}
        z_{i+1}(\vx_{t_i})=A_{i+1},
        \qquad
        x_{t_i}^{(j)}=0
        \quad (j\ge i+2).
        \label{eq:big_step_dynamics}
    \end{align}
    In particular, $\vx_{t_i}$ is in its $(i+1)$-st nonzero affine
    regime and is inactive at every component $j\ge i+2$.
\end{restatable}
\begin{proof}[Proof of \cref{lem:BigStepActivatesNext}]
    Recall from the results in \cref{lem:GapsStayAffine} that, immediately before the step $t_i$,
    \begin{align*}
        z_i(\vx_{t_i - 1}) \ge \delta_i, \qquad
        x_{t_i - 1}^{(i+1)} = \frac{s_i \delta_i }{4} = c_{i+1}, \qquad
        x_{t_i - 1}^{(j)} = 0 \quad (j \ge i+2).
    \end{align*}
    So if we apply \cref{lem:OneStepGD}~\ref{item:margin_one_step_edge} with step size $\eta_{t_i}=b_i$, we have
    \begin{align*}
        x_{t_i}^{(i+1)}=c_{i+1}+\frac{b_i\delta_i}{4},
        \qquad
        x_{t_i}^{(j)}=0 \quad (j\ge i+2).
    \end{align*}
    It follows from the definition of the $(i+1)$-st margin and
    \cref{eq:param_recursive_again} that
    \begin{align*}
        z_{i+1}(\vx_{t_i})
        &=x_{t_i}^{(i+1)}-x_{t_i}^{(i+2)}-c_{i+1}
        =\frac{b_i\delta_i}{4}
        =A_{i+1}.
    \end{align*}
    This proves \cref{eq:big_step_dynamics}.
    We now verify the claimed regimes. By our parameter choice (\cref{eq:param_recursive_again}),
    \begin{align*}
        \delta_{i+1} =
        \begin{dcases}
            \dfrac{2A_{i+1}}{s_{i+1}+2}, & i<k;\\ \dfrac{A_{k+1}}{s_{k+1}+1}, & i=k;
        \end{dcases}
        \qquad\implies\qquad
        \delta_{i+1}\le A_{i+1}.
    \end{align*}
    Hence
    $z_{i+1}(\vx_{t_i})=A_{i+1}\ge\delta_{i+1}$, so component $i+1$ is in its nonzero affine regime. For every $j\ge i+2$, all coordinates from $j$ onward are zero, and thus \(z_j(\vx_{t_i})=-c_j\le0 \). Therefore, all components $j\ge i+2$ remain inactive.
\end{proof}

\subsection{\texorpdfstring{Proof of \cref{lem:GeneralLowerBound}}{Proof of Lem 3.2}}
\label{subsec:general_lowerbound_proof}

Now, by applying \cref{lem:GapsStayAffine,lem:BigStepActivatesNext} alternately and dealing with the last ($(k+1)$-st) block of small step sizes, we finally obtain a general lower bound of the function value gap as in \cref{lem:GeneralLowerBound} below.

\lemGeneralLowerBound*
\begin{proof}[Proof of \cref{lem:GeneralLowerBound}]
    Let us first establish the $i$-th margin and zero coordinates at the beginning of the $i$-th block for every $i\in[k+1]$: the margin equals $z_i(\vx_{t_{i-1}})=A_i$ and $j$-th component is $x_{t_{i-1}}^{(j)}=0$ ($i+1 \le j \le k+1$).
    For $i=1$, recall that $\vx_0=r_0\ve_1$, $A_1=r_0$, and $c_1=0$. Thus $z_1(\vx_0)=r_0=A_1$ and $x_0^{(j)}=0$ ($j\ge2$); the claim holds at the start of the first block.
    Also, if the claim holds for $i\in[k]$, then applying \cref{lem:GapsStayAffine,lem:BigStepActivatesNext} together shows the same claim for $i+1$.
    Hence, an induction over $i\in[k+1]$ proves the claim.
    In particular, it holds that $z_{k+1}(\vx_{t_k}) = A_{k+1} \ge \delta_{k+1}$, thereby being able to apply \cref{lem:OneStepGD}~\ref{item:margin_one_step_lower_bound} for $(k+1)$-st margins.
    
    Now, we handle the terminal ($(k+1)$-st) block $s_{k+1}$.
    For brevity, write $A=A_{k+1}$, $s=s_{k+1}$, and $\delta=\delta_{k+1}=A/(s+1)$. Consider any iterate in the terminal block, and let $\sigma_t:=\sum_{\tau=t_k+1}^{t}\eta_\tau$ for $t=t_k,\ldots,N$. Since $\sigma_t\le s$, repeated application of \cref{lem:OneStepGD}~\ref{item:margin_one_step_lower_bound} yields
    \begin{align*}
        z_{k+1}(\vx_t)
        &\ge A-\frac{\sigma_t\delta}{2} \ge A-\frac{s\delta}{2} = \frac{A(s+2)}{2(s+1)} \ge\frac{A}{s+1} =\delta.
    \end{align*}
    Thus the last component stays in its nonzero affine regime throughout the terminal block. In particular, the margin of the last iterate satisfies $z_{k+1}(\vx_N) \ge A-\frac{s\delta}{2}$.

    Since $F$ is nonnegative componentwise and the last Huber term appears with coefficient $1/2$, we may keep only that component. Using the affine expression $H_\delta(z)=\delta z-\delta^2/2$ valid for $z\ge\delta$, we obtain
     \begin{align*}
        F(\vx_N)
        &\ge \frac{1}{2}H_\delta\open{z_{k+1}(\vx_N)} \\
        &=\frac{\delta}{2}\, z_{k+1}(\vx_N)-\frac{\delta^2}{4}\\
        &\ge \frac{\delta}{2}\left(A-\frac{s\delta}{2}\right)-\frac{\delta^2}{4} \\
        &=\frac{\delta A}{2}-\frac{(s+1)\delta^2}{4}
        =\frac{A^2}{4(s+1)}.
    \end{align*}
    Therefore, substituting the explicit formula for $A_{k+1}$ in~\cref{eq:parameter_A_i_rolledout} directly proves \cref{eq:general_function_value_lower_bound}.
    
    Lastly, recall from~\cref{prop:HardInstanceProperties} that $F \in \mathscr{F}_{1}(\R^{k+1})$ and $\vzero \in X_F^{\star}$.
    Since the convergence rate~\cref{eq:convergence_rate_R_N} is defined as a supremum, dividing by $\norm{\vx_0-\vzero}_2^2=r_0^2$ finally yields the lower bound in \cref{eq:general_normalized_lower_bound}.
\end{proof}

We remark that the lower bound above applies to an arbitrary selection of $\bigset{t_1, \cdots, t_k} \subset [N]$ as well as their count ($k$), and its strength depends critically on that selection. 
Later, we optimize this general lower bound over all choices of big-step locations to obtain our convergence rate lower bounds.

%% file: sec/903NonAnytimeProofs.tex
\section{Proofs: Non-Anytime Lower Bound}
\label{sec:nonanytime_proofs}

\subsection{\texorpdfstring{Proof of \cref{eq:all_selection_objective_reformulated}}{Proof of Eq. (13)}}
\label{subsec:TerminalParameterIdentity}

    For $s\ge0$ and $\lambda\ge2$, let $\chi_s(\lambda):=4(\lambda-1)/(\lambda+s)^2$. Its derivative is
    \begin{align*}
        \chi_s'(\lambda)
        =\frac{4(\lambda+s)-8(\lambda-1)}{(\lambda+s)^{3}}=\frac{4(s+2-\lambda)}{(\lambda+s)^3},
    \end{align*}
    so $\chi_s$ is maximized over $\lambda\ge2$ at $\lambda=s+2\ge2$, with
    \begin{align*}
        \max_{\lambda\ge2}\chi_s(\lambda)=\frac{4(s+1)}{(2s+2)^2}=\frac{1}{1+s}.
    \end{align*}
    Using $s=s_{\abs{T}+1}(T;\veta)$, we derive
    \begin{align*}
        \max_{\lambda\ge2}\frac{4(\lambda-1)}{\gV_\lambda(\veta)^2}
        &=\max_{\lambda\ge2}\max_{T\subseteq[N]}\frac{4(\lambda-1)\,\gP(T;\veta)^2}{\open{\lambda+s_{\abs{T}+1}(T;\veta)}^2}\\
        &=\max_{T\subseteq[N]}\gP(T;\veta)^2\max_{\lambda\ge2}\frac{4(\lambda-1)}{\open{\lambda+s_{\abs{T}+1}(T;\veta)}^2}\\
        &=\max_{T\subseteq[N]}\frac{\gP(T;\veta)^2}{1+s_{\abs{T}+1}(T;\veta)}
        =\gB(\veta),
    \end{align*}
    proving the identity.

\subsection{\texorpdfstring{Proof of \cref{lem:ExactCostRecursion}}{Proof of Lem. 3.1}}
\label{subsec:proof_cost_recursion}

We restate the lemma for readability.

\lemExactCostRecursion*
\begin{proof}[Proof of \cref{lem:ExactCostRecursion}]
    The first term in \cref{eq:exact_cost_recursion} is the cost $\Psi_\lambda(\emptyset;\veta_{1:n})$ associated with the empty checkpoint set.
    Now fix a nonempty checkpoint set $T$ and let $t\in T$. Split $T=T_L\cup\{t\}\cup T_R$ with $T_L\subseteq[1,t-1]$ and $T_R\subseteq[t+1,n]$. The gaps of $T$ strictly to the left of $t$ are exactly the gaps of $T_L$ inside $\veta_{1:t-1}$, the last of them being the terminal gap of $T_L$, the gaps strictly to the right are exactly the gaps of $T_R$ inside $\veta_{t+1:n}$, and the factor attached to the checkpoint index $t$ is $2/\eta_t$. Therefore
    \begin{align}
        \Psi_\lambda(T;\veta_{1:n})
        &=
        \frac{2}{\eta_t}\cdot
        \frac{2+s_{\abs{T_L}+1}(T_L;\veta_{1:t-1})}{\gP(T_L;\veta_{1:t-1})}
        \cdot
        \Psi_\lambda(T_R;\veta_{t+1:n})\nonumber\\
        &=\frac{2}{\eta_t}\cdot
        \Psi_2(T_L;\veta_{1:t-1})
        \cdot
        \Psi_\lambda(T_R;\veta_{t+1:n}).
        \label{eq:cost_factorization}
    \end{align}
    Minimizing \cref{eq:cost_factorization} over $T_L$ and $T_R$ independently shows that the minimum of $\Psi_\lambda(T;\veta_{1:n})$ over all checkpoint sets $T$ containing $t$ equals $\frac{2}{\eta_t}\gV_2(\veta_{1:t-1})\gV_\lambda(\veta_{t+1:n})$. Conversely, joining minimizers of the two costs $\gV_2(\veta_{1:t-1})$ and $\gV_\lambda(\veta_{t+1:n})$ with the index $t$ produces a checkpoint set attaining this value. Taking the minimum over $t\in[n]$ ranges over all nonempty checkpoint sets, and combining with the empty checkpoint set proves \cref{eq:exact_cost_recursion}.
\end{proof}

\subsection{Existence and Strict Growth of Maximal Schedules}
\label{subsec:proof_extremal_schedule}

\begin{lemma}
    \label{lem:ExtremalSchedule}
        For every $n\ge 0$ and $\lambda\ge 2$, the supremum defining $\gU_n(\lambda)$ is finite and attained. Moreover, the sequence $n\mapsto \gU_n(\lambda)$ is strictly increasing:
        \begin{align}
            \gU_n(\lambda)>\gU_{n-1}(\lambda),\qquad n\ge 1.
            \label{eq:worst_case_cost_strict_growth}
        \end{align}
    \end{lemma}
\begin{proof}[Proof of \cref{lem:ExtremalSchedule}]
    We proceed by induction on $n$, the induction hypothesis being that $\gU_{n'}(\lambda)$ is attained for every $n'<n$ and that $\gU_0(\lambda)<\cdots<\gU_{n-1}(\lambda)$. The claim is immediate for $n=0$.

    \medskip

    We first establish the strict growth of $\gU_n(\lambda)$ in $n$. By the induction hypothesis, $\gU_{n-1}(\lambda)$ is attained by some positive schedule $\xi\in(0,\infty)^{n-1}$. Append a new step $\epsilon>0$ to the schedule and denote the resulting length-$n$ schedule by
    \[
        \xi^{(\epsilon)}:=(\xi_1,\ldots,\xi_{n-1},\epsilon).
    \]
    If the new step is absent from a checkpoint set, then $\epsilon$ is added to the terminal gap. Since $\Psi_\lambda$ is strictly increasing in every gap sum, the resulting cost is strictly higher than the cost of the corresponding checkpoint set for $\xi$, hence strictly larger than $\gV_\lambda(\xi)$. If the new step is a checkpoint, the cost contains the factor $2/\epsilon$ and therefore tends to $+\infty$ as $\epsilon\downarrow0$. Since there are only finitely many checkpoint sets, for sufficiently small $\epsilon>0$ every associated cost is strictly higher than $\gV_\lambda(\xi)=\gU_{n-1}(\lambda)$. Hence \cref{eq:worst_case_cost_strict_growth} follows.

    \medskip

    We next prove finiteness. For an arbitrary positive schedule $\veta\in(0,\infty)^n$, consider the checkpoint set $T_{\ge 1}:=\{t:\eta_t\ge1\}$. Every gap of $T_{\ge 1}$ consists of steps smaller than $1$, so each gap sum is at most $n$. Moreover $\abs {T_{\ge 1}}\le n$ and $b_i(T_{\ge 1};\veta)\ge 1$ for every $i$. Hence
    \begin{align}
        \gV_\lambda(\veta)
        \le \Psi_\lambda(T_{\ge1};\veta)
        \le (\lambda+n)[2(n+2)]^{\abs{T_{\ge1}}}
        \le (\lambda+n)[2(n+2)]^n<\infty.
        \label{eq:worst_case_cost_finite}
    \end{align}
    so $\gU_n(\lambda)<\infty$.

    \medskip

    To prove attainment, take a maximizing sequence
    $\veta^{(j)}\in(0,\infty)^n$ such that
    $\gV_\lambda(\veta^{(j)})\to\gU_n(\lambda)$,
    and put $M^{(j)}:=\max_i\eta_i^{(j)}$. If $M^{(j)}\ge 1$, then $T_{\ge 1}$ contains an index attaining $M^{(j)}$. We obtain a stronger bound than \cref{eq:worst_case_cost_finite}:
    \begin{align*}
        \gV_\lambda(\veta^{(j)})\le\frac{(\lambda+n)\,[2(n+2)]^n}{M^{(j)}}.
    \end{align*}
    Thus $M^{(j)}\to\infty$ along a subsequence would force $\gV_\lambda(\veta^{(j)})$ to tend to $0$. Since $\veta^{(j)}$ is a maximizing sequence,
    \[
        \gV_\lambda(\veta^{(j)})\longrightarrow\gU_n(\lambda)\ge\gU_0(\lambda)
        =\lambda>0.
    \]
    Hence $M^{(j)}$ cannot diverge to $+\infty$ along any subsequence. Therefore $M^{(j)}\le \bar M$ for some finite $\bar M$ and all $j$, and after passing to a subsequence we may assume $\veta^{(j)}\to\bar\veta\in[0,\bar M]^n$.

    Suppose that one or more coordinates of $\bar\veta$ are zero. Delete them and let $\bar \veta^+$ be the remaining positive schedule, of length $n'<n$. Choose a checkpoint set minimizing $\Psi_\lambda(\cdot;\bar\veta^+)$ and use the corresponding indices in $\veta^{(j)}$, leaving all coordinates converging to zero outside the checkpoint set. The checkpoint steps converge to positive limits, so no factor blows up, and the contributions of the vanishing coordinates to the gap tend to $0$. Therefore
    \begin{align*}
        \gU_n(\lambda)=\lim_{j\to\infty}\gV_\lambda(\veta^{(j)})\le\gV_\lambda(\bar\veta^+)\le\gU_{n'}(\lambda)\le\gU_{n-1}(\lambda),
    \end{align*}
    where the last inequality uses the induction hypothesis that $\gU_\cdot(\lambda)$ is increasing below level $n$. This contradicts \cref{eq:worst_case_cost_strict_growth}. Hence $\bar\veta$ lies in the positive orthant.

    \medskip

    Since $\gV_\lambda$ is the minimum of finitely many checkpoint-set costs, each continuous on $(0, \infty)^n$, it is continuous there. Therefore $\gV_\lambda(\bar\veta)=\lim_j\gV_\lambda(\veta^{(j)})=\gU_n(\lambda)$, proving attainment and completing the induction.
\end{proof}

\subsection{Log-submodularity of Cost Function}
\label{subsec:proof_logsubmodularity}

\begin{lemma}
    \label{lem:CostLogSubmodularity}
        Fix $\lambda\ge2$ and a positive schedule $\veta\in(0,\infty)^n$. For any two checkpoint sets $T_1,T_2\subseteq[n]$,
        \begin{align}
            \Psi_\lambda(T_1;\veta)\Psi_\lambda(T_2;\veta)
            \ge
            \Psi_\lambda(T_1\cap T_2;\veta)\Psi_\lambda(T_1\cup T_2;\veta).
            \label{eq:cost_log_submodularity}
        \end{align}
        If $T_1$ and $T_2$ are nonempty and disjoint, then the inequality is strict.
    \end{lemma}
\begin{proof}[Proof of \cref{lem:CostLogSubmodularity}]
    It suffices to prove diminishing marginal ratios. Fix a checkpoint set $T\subseteq[n]$ and an index $t\notin T$. Within the gap of $T$ containing $t$, let $\ell$ be the sum of the steps outside $T$ between the nearest checkpoint index to the left of $t$ (or the start of the schedule) and $t$. Let $\rho$ be the sum of the steps outside $T$ between $t$ and the right boundary of that gap, and set
    \begin{align*}
        K:=\begin{cases}
            2+\rho, & \text{if $T$ has a checkpoint index to the right of $t$},\\
            \lambda+\rho, & \text{otherwise}.
        \end{cases}
    \end{align*}
    In particular, $K\ge2$.

    Consider adding $t$ to $T$. If the gap containing $t$ is followed by a checkpoint index, then, after canceling all unchanged factors, the original contribution from this gap is replaced according to
    \[
        2(2+\ell+\eta_t+\rho)
        \quad\longmapsto\quad
        \frac{2(2+\ell)}{\eta_t}\,2(2+\rho).
    \]
    If the gap is terminal, then
    \[
        \lambda+\ell+\eta_t+\rho
        \quad\longmapsto\quad
        \frac{2(2+\ell)}{\eta_t}\,(\lambda+\rho).
    \]
    Thus, in both cases,
    \begin{align}
        \frac{\Psi_\lambda(T\cup\{t\};\veta)}{\Psi_\lambda(T;\veta)}
        =
        \frac{2(2+\ell)K}{\eta_t(\ell+\eta_t+K)}.
        \label{eq:cost_marginal_ratio}
    \end{align}
    Differentiating the logarithm of the right-hand side with respect to $\ell$ and $K$ gives
    \begin{align*}
        \frac{1}{2+\ell}-\frac{1}{\ell+\eta_t+K}
        &=
        \frac{\eta_t+K-2}{(2+\ell)(\ell+\eta_t+K)}
        >0, \nonumber\\
        \frac{1}{K}-\frac{1}{\ell+\eta_t+K}
        &=
        \frac{\ell+\eta_t}{K(\ell+\eta_t+K)}
        >0,
    \end{align*}
    since $K\ge2$ and $\eta_t>0$. Hence the marginal ratio in \cref{eq:cost_marginal_ratio} is strictly increasing in both $\ell$ and $K$.

    Now let $T\subseteq T'$ and $t\notin T'$. The additional checkpoint indices in $T'\setminus T$ can only move the two boundaries of the gap containing $t$ toward $t$. In particular, $\ell_{T'}\le\ell_T$. For the right boundary, if the gap remains of the same type, then $\rho_{T'}\le\rho_T$, and hence $K_{T'}\le K_T$. If a terminal gap for $T$ becomes an interior gap for $T'$, then $K_{T'}=2+\rho_{T'}\le 2+\rho_T\le \lambda+\rho_T=K_T$.
    Therefore the monotonicity of \cref{eq:cost_marginal_ratio} gives
    \begin{align}
        \frac{\Psi_\lambda(T\cup\{t\};\veta)}{\Psi_\lambda(T;\veta)}
        \ge
        \frac{\Psi_\lambda(T'\cup\{t\};\veta)}{\Psi_\lambda(T';\veta)}.
        \label{eq:cost_diminishing_marginal}
    \end{align}

    We now apply \cref{eq:cost_diminishing_marginal} to arbitrary $T_1,T_2\subseteq[n]$. Write
    \[
        T_2\setminus T_1=\{i_1,\ldots,i_L\},
    \]
    in any order, and for $j=1,\ldots,L$ define
    \[
        X_j:=(T_1\cap T_2)\cup\{i_1,\ldots,i_{j-1}\},
        \qquad
        X'_j:=T_1\cup\{i_1,\ldots,i_{j-1}\}.
    \]
    Then $X_j\subseteq X'_j$ and $i_j\notin X'_j$, so \cref{eq:cost_diminishing_marginal} gives
    \begin{align*}
        \frac{\Psi_\lambda(X_j\cup\{i_j\};\veta)}
             {\Psi_\lambda(X_j;\veta)}
        \ge
        \frac{\Psi_\lambda(X'_j\cup\{i_j\};\veta)}
             {\Psi_\lambda(X'_j;\veta)}.
    \end{align*}
    Multiplying these inequalities over $j=1,\ldots,L$ and telescoping yields
    \begin{align*}
        \frac{\Psi_\lambda(T_2;\veta)}
             {\Psi_\lambda(T_1\cap T_2;\veta)}
        \ge
        \frac{\Psi_\lambda(T_1\cup T_2;\veta)}
             {\Psi_\lambda(T_1;\veta)},
    \end{align*}
    which is equivalent to \cref{eq:cost_log_submodularity}.

    We still need to prove strictness when $T_1$ and $T_2$ are nonempty and disjoint. Add the elements of $T_2$ in any order to $\emptyset$ and to $T_1$ in parallel. Consider the first added element $t\in T_2$. Since $T_1$ is nonempty and $t\notin T_1$, at least one boundary parameter for the gap containing $t$ is strictly smaller under $T_1$ than under $\emptyset$.

    Indeed, if $T_1$ contains an index to the left of $t$, then the nearest such checkpoint index and its positive step-size are excluded from the left gap sum, so
    \begin{align*}
        \ell_{T_1}<\ell_{\emptyset}.
    \end{align*}
    Otherwise, every element of $T_1$ lies to the right of $t$. In this case, the gap is terminal under $\emptyset$ but has a checkpoint right boundary under $T_1$. If $t'>t$ is the nearest checkpoint index in $T_1$, then
    \begin{align*}
        K_{T_1}
        =
        2+\sum_{i=t+1}^{t'-1}\eta_i
        <
        \lambda+\sum_{i=t+1}^{n}\eta_i
        =
        K_{\emptyset},
    \end{align*}
    where the inequality is strict because $\eta_{t'}>0$ and $\lambda\ge2$.

    Thus at least one of the two parameters $\ell$ and $K$ decreases strictly, while neither increases. Since the marginal ratio in \cref{eq:cost_marginal_ratio} is strictly increasing in both parameters, the first marginal comparison is strict. Multiplying the marginal inequalities as above therefore gives
    \begin{align*}
        \Psi_\lambda(T_1;\veta)\Psi_\lambda(T_2;\veta)
        >
        \Psi_\lambda(\emptyset;\veta)\Psi_\lambda(T_1\cup T_2;\veta),
    \end{align*}
    which is precisely the strict form of \cref{eq:cost_log_submodularity}.
\end{proof}

\subsection{\texorpdfstring{Properties of the composition map $\gQ$}{Properties of the composition map Q}}
\label{subsec:lemmas_Q}

Recall that
\begin{align*}
    \gQ(x,y):=\frac{x+y-2+\sqrt{(x+y-2)^2+8xy}}{2}.
\end{align*}
It is the unique positive solution of
\begin{align}
    w(w-x-y+2)=2xy.
    \label{eq:binary_composition_equation}
\end{align}
Moreover,
\begin{align*}
    \alpha_\star:=\log_2(1+\sqrt3)=1.449984313\ldots, \qquad
    \nu_\star:=\frac{1}{\alpha_\star}=0.689662633\ldots.
\end{align*}

\subsubsection{Elementary properties}

\begin{lemma}[Elementary properties of $\gQ$]
\label{lem:QProperties}
    For all $x,y\ge 2$, the value $w=\gQ(x,y)$ satisfies $w>\max\{x,y\}\ge 2$, and $\gQ$ is strictly increasing in each argument.
\end{lemma}
\begin{proof}[Proof of \cref{lem:QProperties}]
    Let $\pi(w):=w^2-(x+y-2)w-2xy$, so that $\gQ(x,y)$ is its larger root. Then $\pi(x)=x^2-x^2-xy+2x-2xy=x(2-3y)<0$ since $y\ge 2$, and symmetrically $\pi(y)<0$. As $\pi$ is an upward parabola, its larger root exceeds $\max\{x,y\}$. Differentiating $\pi(w)=0$ implicitly in $x$ gives
    \[
    (2w-x-y+2)\,\partial_x w=w+2y.
    \]
    From \cref{eq:binary_composition_equation} and $w>0$ we get $w-x-y+2=2xy/w>0$, so $2w-x-y+2>w>0$. Hence $\partial_x w>0$, and symmetrically $\partial_y w>0$.
\end{proof}

\subsubsection{Power-subadditivity}

\begin{lemma}[Power-subadditivity of $\gQ$]
    \label{lem:QPowerSubadditivity}
        For all $x, y\ge 2$,
        \begin{align}
            \gQ(x,y)^{\nu_\star}<x^{\nu_\star}+y^{\nu_\star}.
            \label{eq:q_power_subadditivity}
        \end{align}
    \end{lemma}
\begin{proof}[Proof of \cref{lem:QPowerSubadditivity}]
    Let $w=\gQ(x,y)$ and normalize by putting $\hat x:=x/w$ and $\hat y:=y/w$, which lie in $(0,1)$ by \cref{lem:QProperties}. Dividing \cref{eq:binary_composition_equation} by $w^2$ gives
    \begin{align}
        \hat x+\hat y+2\hat x\hat y=1+\frac{2}{w}>1.
        \label{eq:composition_normalized_identity}
    \end{align}
    We show that $\hat x^{\nu_\star}+\hat y^{\nu_\star}>1$, which is \cref{eq:q_power_subadditivity} after multiplying by $w^{\nu_\star}$.

    Suppose instead that $\sigma:=\hat x^{\nu_\star}+\hat y^{\nu_\star}\le1$, and set $u:=\hat x^{\nu_\star}/\sigma\in[0,1]$. Since $1/\nu_\star=\alpha_\star$,
    \begin{align*}
        \hat x=\sigma^{\alpha_\star}u^{\alpha_\star},\qquad
        \hat y=\sigma^{\alpha_\star}(1-u)^{\alpha_\star}.
    \end{align*}
    Hence, using $\sigma\le1$,
    \begin{align}
        \hat x+\hat y+2\hat x\hat y
        &=
        \sigma^{\alpha_\star}\left[u^{\alpha_\star}+(1-u)^{\alpha_\star}
        +2\sigma^{\alpha_\star}[u(1-u)]^{\alpha_\star}\right]
        \ \le\
        \sigma^{\alpha_\star}\gG(u),
        \label{eq:composition_g_reduction}
    \end{align}
    where $\gG(u):=u^{\alpha_\star}+(1-u)^{\alpha_\star}+2[u(1-u)]^{\alpha_\star}$.

    We claim that $\gG\le1$ on $[0,1]$. By symmetry, it suffices to consider $u\in[0,1/2]$. Put $\bar\alpha:=\alpha_\star-1\in(0,1/2)$. Differentiation and factoring out $\alpha_\star u^{\bar\alpha}(1-u)^{\bar\alpha}>0$ gives
    \begin{align*}
        \gG'(u)=\alpha_\star u^{\bar\alpha}(1-u)^{\bar\alpha}\varphi(u),\qquad
        \varphi(u):=(1-u)^{-\bar\alpha}-u^{-\bar\alpha}+2(1-2u).
    \end{align*}
    Moreover,
    \begin{align*}
        \varphi''(u)=\bar\alpha(\bar\alpha+1)\left((1-u)^{-\bar\alpha-2}-u^{-\bar\alpha-2}\right)<0,\qquad 0<u<\frac12,
    \end{align*}
    so $\varphi$ is strictly concave on $(0,1/2)$. We also have $\varphi(0+)=-\infty$, $\varphi(1/2)=0$, and
    \begin{align*}
        \varphi'(1/2)=\bar\alpha\,2^{\bar\alpha+2}-4<0,
    \end{align*}
    since $\bar\alpha<1/2$ gives $\bar\alpha2^{\bar\alpha+2}<\tfrac12\cdot2^{5/2}=2\sqrt2<4$. A strictly concave function with $\varphi(0+)=-\infty$, $\varphi(1/2)=0$ and $\varphi'(1/2)<0$ is negative on an initial interval and positive afterwards. Hence $\gG$ first decreases and then increases on $[0,1/2]$, and its maximum there is attained at an endpoint. Clearly $\gG(0)=1$, while with $2^{\alpha_\star}=1+\sqrt3$,
    \begin{align*}
        \gG(1/2)=\frac{2}{2^{\alpha_\star}}+\frac{2}{(2^{\alpha_\star})^2}=\frac{2\cdot2^{\alpha_\star}+2}{(2^{\alpha_\star})^2}=\frac{4+2\sqrt 3}{4+2\sqrt 3}.
    \end{align*}
    Thus $\gG(u)\le1$ on $[0,1]$.

    Returning to \cref{eq:composition_g_reduction}, we obtain $\hat x+\hat y+2\hat x\hat y\le\sigma^{\alpha_\star}\le1$, contradicting \cref{eq:composition_normalized_identity}. Therefore $\sigma>1$, as claimed.
\end{proof}

\subsection{\texorpdfstring{Proof of \cref{lem:WorstCaseCostBound}}{Proof of Lem. 3.2}}
\label{subsec:proof_worst_case_cost}

We restate the lemma for readability.

\lemWorstCaseCostBound*

The proof now follows the sketch established in \cref{sec:nonanytime}.
We first use the extremal-schedule and log-submodularity lemmas, together with monotonicity of $\gQ$, to prove the recursive estimate in \cref{eq:recursive_worst_case_cost_bound}.
We then combine that estimate with power-subadditivity of $\gQ$ to prove \cref{eq:uniform_worst_case_cost_bound}.

\begin{proof}[Proof of \cref{lem:WorstCaseCostBound}]
    By \cref{lem:ExtremalSchedule}, choose a positive maximizer $\veta^\star\in(0,\infty)^n$ and write
    \begin{align*}
        w:=\gV_\lambda(\veta^\star)=\gU_n(\lambda).
    \end{align*}
    We call a checkpoint set $T\subseteq[n]$ \emph{tight for $\veta^\star$} if
    \[
        \Psi_\lambda(T;\veta^\star)=w,
    \]
    and write
    \begin{align*}
        \gM:=\{T\subseteq[n]:\Psi_\lambda(T;\veta^\star)=w\}
    \end{align*}
    for the family of tight checkpoint sets for $\veta^\star$. Throughout the proof, we use the fact that there are finitely many checkpoint sets and that the cost associated with each checkpoint set is continuous in $\veta$. Hence, under a sufficiently small perturbation of $\veta^\star$, the cost of every non-tight checkpoint set retains positive slack above $w$.

    \medskip

    We first show that $\gM$ is closed under intersection and union. By \cref{lem:CostLogSubmodularity}, the cost function $\Psi_\lambda(T;\veta)$ is log-submodular in $T$.

    Thus, if $T_1,T_2\in\gM$, then
    \begin{align*}
        w^2
        &=
        \Psi_\lambda(T_1;\veta^\star)\Psi_\lambda(T_2;\veta^\star) \ge
        \Psi_\lambda(T_1\cap T_2;\veta^\star)
        \Psi_\lambda(T_1\cup T_2;\veta^\star)
        \ge
        w^2,
    \end{align*}
    where the last inequality follows from the definition of $w$ as the minimum checkpoint-set cost. It implies that both inequalities are equalities. Hence $T_1\cap T_2,T_1\cup T_2\in\gM$. Moreover, two nonempty disjoint checkpoint sets cannot both be tight for $\veta^\star$; otherwise, the strict part of \cref{lem:CostLogSubmodularity} would make the first inequality strict.

    \medskip

    We next show that the empty checkpoint set is tight for $\veta^\star$. Suppose otherwise. Since $\gM$ is finite and closed under intersections, the set
    \begin{align*}
        T_{\cap}:=\bigcap_{T\in\gM}T
    \end{align*}
    is a tight checkpoint set for $\veta^\star$ and is nonempty. Multiply every step $\eta_i^\star$ with $i\in T_{\cap}$ by a common factor $1-\epsilon$, where $\epsilon>0$ is sufficiently small. Every tight checkpoint set for $\veta^\star$ contains all indices of $T_{\cap}$. Hence each perturbed coordinate is a checkpoint step in every such cost: it appears only through a reciprocal checkpoint-step factor and does not contribute to any gap sum. Decreasing these coordinates therefore strictly increases the cost of every tight checkpoint set for $\veta^\star$. Since the costs of all non-tight checkpoint sets retain positive slack above $w$ for sufficiently small $\epsilon$, every checkpoint-set cost for the perturbed schedule is strictly higher than $w$. Thus the perturbed schedule $\widetilde\veta$ satisfies
    \[
        \gV_\lambda(\widetilde\veta)>w=\gU_n(\lambda),
    \]
    contradicting the definition of $\gU_n(\lambda)$. Hence $\emptyset\in\gM$.

    \medskip

    There must also be a nonempty tight checkpoint set for $\veta^\star$. Otherwise, $\emptyset$ is the unique tight checkpoint set for $\veta^\star$. Increasing one coordinate $\eta_t^\star$ by a sufficiently small amount strictly increases the empty checkpoint-set cost
    \[
        \Psi_\lambda(\emptyset;\veta^\star)
        =
        \lambda+\sum_{i=1}^n\eta_i^\star,
    \]
    while the cost of every non-tight checkpoint set remains above $w$ by continuity.
    This again produces a schedule whose $\gV_\lambda$ is strictly larger than $\gU_n(\lambda)$, a contradiction. Therefore $\gM$ contains a nonempty tight checkpoint set for $\veta^\star$.

    \medskip

    Choose $T_{\min}\in\gM$ inclusion-minimal among the nonempty tight checkpoint sets for $\veta^\star$. Every other nonempty tight checkpoint set $T$ for $\veta^\star$ intersects $T_{\min}$ by the closure and strictness established above. Since $T\cap T_{\min}$ is also a tight checkpoint set for $\veta^\star$ and is a nonempty subset of $T_{\min}$, minimality gives
    \[
        T\cap T_{\min}=T_{\min}.
    \]
    Therefore $T_{\min}\subseteq T$ for every nonempty tight checkpoint set $T$ for $\veta^\star$. Suppose $\abs{T_{\min}}\ge2$. We perturb two coordinates indexed by $T_{\min}$ so that their sum increases while their product decreases. Let $\eta^\star_{\min}$ and $\eta^\star_{\max}$ denote the smallest and largest values among $\{\eta_i^\star:i\in T_{\min}\}$.

    Choose indices attaining these two values. If $\eta^\star_{\min}<\eta^\star_{\max}$, replace these two coordinates by
    \[
        \eta^\star_{\min}-\epsilon,
        \qquad
        \eta^\star_{\max}+\gamma\epsilon,
    \]
    where
    \[
        1<\gamma<\frac{\eta^\star_{\max}}{\eta^\star_{\min}}.
    \]
    For all sufficiently small $\epsilon>0$, the sum increases by $(\gamma-1)\epsilon>0$, whereas the product changes by
    \[
        \epsilon\bigl(\gamma\eta^\star_{\min}-\eta^\star_{\max}\bigr)-\gamma\epsilon^2<0.
    \]

    If all coordinates indexed by $T_{\min}$ are equal to some $\rho>0$, replace two of them by
    \[
        \rho-\epsilon,
        \qquad
        \rho+\epsilon+\frac{\epsilon^2}{2\rho}.
    \]
    Their sum increases by $\epsilon^2/(2\rho)>0$, while their product becomes
    \begin{align*}
        (\rho-\epsilon) \left(\rho+\epsilon+\frac{\epsilon^2}{2\rho}\right)
        =
        \rho^2-\frac{\epsilon^2}{2} -\frac{\epsilon^3}{2\rho}
        <
        \rho^2.
    \end{align*}

    Under either perturbation, the empty checkpoint-set cost $\lambda+\sum_{i=1}^n\eta_i^\star$ strictly increases. Every nonempty tight checkpoint set for $\veta^\star$ contains all indices of $T_{\min}$. Hence the dependence of the cost of every nonempty tight checkpoint set for $\veta^\star$ on the two perturbed coordinates is only through the reciprocal of their product, so every such cost also strictly increases. The costs of all non-tight checkpoint sets retain positive slack above $w$ for a sufficiently small perturbation, contradicting maximality. Therefore
    \[
        T_{\min}=\{\tau\}
    \]
    for some $\tau\in[n]$.

    \medskip

    Since both $\emptyset$ and $\{\tau\}$ are tight checkpoint sets for $\veta^\star$, define
    \begin{align*}
        x:=2+\sum_{i<\tau}\eta_i^\star,\qquad
        y:=\lambda+\sum_{i>\tau}\eta_i^\star,\qquad
        \mu:=\eta_\tau^\star.
    \end{align*}
    Then $x\ge2$ and $y\ge\lambda\ge2$. Their costs satisfy
    \begin{align}
        w
        &=
        \Psi_\lambda(\emptyset;\veta^\star)
        =
        x+y+\mu-2, \nonumber\\
        w
        &=
        \Psi_\lambda(\{\tau\};\veta^\star)
        =
        \frac{2xy}{\mu}.
        \label{eq:tight_empty_singleton_balance}
    \end{align}
    Eliminating $\mu$ from these two identities gives
    \[
        w^2=(x+y-2)w+2xy \quad \implies \quad w(w-x-y+2)=2xy.
    \]
    Hence $w=\gQ(x,y)$.

    \medskip

    We now obtain the recursive estimate. Apply \cref{lem:ExactCostRecursion} at the index $\tau$. Since the recursion takes the minimum over all choices of the splitting index, the candidate corresponding to $\tau$ must have cost at least $w$. Therefore
    \begin{align*}
        w
        \le
        \frac{2}{\mu}
        \gV_2(\veta^\star_{1:\tau-1})
        \gV_\lambda(\veta^\star_{\tau+1:n}).
    \end{align*}
    On the other hand, each child term is bounded above by its empty checkpoint-set cost:
    \begin{align*}
        \gV_2(\veta^\star_{1:\tau-1})\le x,
        \qquad
        \gV_\lambda(\veta^\star_{\tau+1:n})\le y.
    \end{align*}
    Hence
    \begin{align*}
        w
        \le
        \frac{2}{\mu}
        \gV_2(\veta^\star_{1:\tau-1})
        \gV_\lambda(\veta^\star_{\tau+1:n})
        \le
        \frac{2xy}{\mu}
        =
        w,
    \end{align*}
    where the last equality follows from \cref{eq:tight_empty_singleton_balance}. Thus both inequalities are equalities. Since the two child terms are positive and are respectively bounded above by $x$ and $y$, equality of their product with $xy$ forces
    \begin{align*}
        \gV_2(\veta^\star_{1:\tau-1})=x,
        \qquad
        \gV_\lambda(\veta^\star_{\tau+1:n})=y.
    \end{align*}
    Writing $i:=\tau-1$, $j:=n-\tau$, we obtain $x\le\gU_i(2)$ and $y\le\gU_j(\lambda)$. Since $\gQ$ is increasing in each argument by \cref{lem:QProperties},
    \begin{align*}
        \gU_n(\lambda)=w=\gQ(x,y)
        \le
        \gQ\bigl(\gU_i(2),\gU_j(\lambda)\bigr).
    \end{align*}
    Maximizing over all $i,j\ge0$ with $i+j=n-1$ proves \cref{eq:recursive_worst_case_cost_bound}.

    \medskip

    To obtain the power bound, we proceed by induction on $n$. The claim holds for $n=0$. For $n\ge1$, \cref{eq:recursive_worst_case_cost_bound,lem:QPowerSubadditivity} and the induction hypothesis give
    \begin{align*}
        \gU_n(\lambda)^{\nu_\star}
        &\le
        \max_{i+j=n-1}\left(\gU_i(2)^{\nu_\star}+\gU_j(\lambda)^{\nu_\star}\right) \nonumber\\
        &\le
        \max_{i+j=n-1}\left(2^{\nu_\star}+2^{\nu_\star}i+\lambda^{\nu_\star}+2^{\nu_\star}j\right)
        =
        \lambda^{\nu_\star}+2^{\nu_\star}n,
    \end{align*}
    proving \cref{eq:uniform_worst_case_cost_bound}.
\end{proof}

%% file: sec/904AnytimeProofs.tex
\section{Proofs: Anytime Lower Bound}
\label{sec:anytime_proofs}

The proof of \cref{lem:ImprovedRecordSumBound} requires \cref{lem:OrderSensitiveCount,lem:OrderSensitiveSumMaximum}. So we first provide their proofs below.

\begin{restatable}[Bound on number of large steps]{lemma}{lemOrderSensitiveCount}
\label{lem:OrderSensitiveCount}
    For $m\ge 1$, let $\vxi=(\xi_1,\ldots,\xi_m)\in (0,\infty)^m$.
    Take any $M\ge\max_{t\in[m]} \xi_t$ and assume $\gV_2(\vxi)\ge\frac{M}{2}$.
    Define the number of step-sizes exceeding $h>0$ by
    \begin{align*}
        \gN_\vxi(h):=\bigl|\{t\in[m]:\xi_t>h\}\bigr|.
    \end{align*}
    Then, for every $0<h\le M$,
    \begin{align*}
        \gN_\vxi(h)+1 \le 4^{\nu_\star}(m+1)h^{-\nu_\star}.
    \end{align*}
\end{restatable}
\begin{proof}[Proof of \cref{lem:OrderSensitiveCount}]
    Fix $h\in(0,M]$.
    Write $J_h:=\{t\in [m]:\xi_t>h\}=\{t_1<\cdots<t_k\}$, $t_0=0$, and $t_{k+1}=m+1$.
    
    Suppose $k\ge 1$.
    The selected indices in $J_h$ partition the remaining entries of $\vxi$ into $k+1$ contiguous blocks $\vxi_{1:t_1-1},\dotsc,\vxi_{t_k+1:m}$, of lengths $m_0,\dotsc,m_k$, respectively.
    Note that $\sum_{j=0}^{k}(m_j+1)=m+1$ because exactly $k$ indices are removed.
    Iterating the factorization in \cref{eq:cost_factorization} at indices $t_1, \dotsc, t_k$ and minimizing each cost separately,
    \begin{align*}
        \gV_2(\vxi)
        \le
        \left(\prod_{i=1}^{k}\frac{2}{\xi_{t_i}}\right)\prod_{j=0}^{k}\gV_2(\vxi_{t_j+1:t_{j+1}-1})
        <
        2^{k}h^{-k} \prod_{j=0}^{k}\gV_2(\vxi_{t_j+1:t_{j+1}-1}),
    \end{align*}    
    where the second inequality uses $\xi_{t_i}>h$.
    Since \cref{lem:WorstCaseCostBound} at $\lambda=2$ implies that
    \begin{align*}
        \gV_2(\vxi_{t_j+1:t_{j+1}-1}) \le 2(m_j+1)^{\alpha_\star}, \qquad (j=0, \dotsc, k)
    \end{align*}
    we have
    \begin{align*}
        \gV_2(\vxi)
        \le
        2^{2k+1}h^{-k}\prod_{j=0}^{k}(m_j+1)^{\alpha_\star}.
    \end{align*}
    Applying the AM-GM inequality to the positive numbers $m_0+1, \dotsc, m_k+1$, we have
    \begin{align*}
        \gV_2(\vxi)
        &\le
        2^{2k+1}h^{-k}\left(\frac{m+1}{k+1}\right)^{\alpha_\star(k+1)}.
    \end{align*}
    Note that this is also true when $k=0$, as \cref{lem:WorstCaseCostBound} at $\lambda=2$ directly implies $\gV_2(\vxi) \le 2(m+1)^{\alpha_\star}$.
    Hence, let $k\ge0$ from now on.
    Apply the assumption that $\gV_2(\vxi)\ge \frac{M}{2}\ge \frac{h}{2}$,
    take the $(k+1)$-st root, and raise to the power $\nu_\star=1/\alpha_\star$. Then, we have
    \begin{align*}
        k+1\le4^{\nu_\star}(m+1)h^{-\nu_\star}.
    \end{align*}
    This proves the desired inequality.
\end{proof}

\begin{restatable}[Sum-maximum bound]{lemma}
{lemOrderSensitiveSumMaximum}
\label{lem:OrderSensitiveSumMaximum}
    Under the same assumption as \cref{lem:OrderSensitiveCount},
    \begin{align*}
        \sum_{i=1}^{m}\xi_i + M \le C\,(m+1)\,M^{1-\nu_\star}, \qquad \text{where } C:=\frac{4^{\nu_\star}}{1-\nu_\star}.
    \end{align*}
\end{restatable}

\begin{proof}[Proof of \cref{lem:OrderSensitiveSumMaximum}]
    Since $0<\xi_t\le M$ for every $t\in [m]$, we can write $\xi_t = \int_0^M \1_{\{h<\xi_t\}} \mathrm{d}h$.
    Summing these integrals for all $t\in [m]$, we have
    \begin{align*}
        \sum_{i=1}^m \xi_i =
        \int_0^{M}\gN_\xi(h)\,dh.
    \end{align*}
    Applying \cref{lem:OrderSensitiveCount} and using $\nu_\star<1$,
    \begin{align*}
        \sum_{i=1}^m \xi_i\le \int_0^{M} 4^{\nu_\star}(m+1) h^{-\nu_\star} -1 \,dh=
        \frac{4^{\nu_\star}}{1-\nu_\star}(m+1)M^{1-\nu_\star}-M.
    \end{align*}
    This proves the lemma.
\end{proof}

\lemImprovedRecordSumBound*
\begin{proof}[Proof of \cref{lem:ImprovedRecordSumBound}]
    Since $n$ is a record time of $\veta$,
    every step of $\veta_{1:n}$ is at most $M_n=\eta_n$.
    Fix any checkpoints $T=\{t_1<\dotsb<t_k\}\subseteq[n-1]$ and take $T_+ := T\cup\{n\}$.
    Observe that 
    \begin{align*}
        s_{k+1}(T_+;\veta_{1:n})=s_{k+1}(T;\veta_{1:n-1}) \quad \text{and} \quad s_{k+2}(T_+;\veta_{1:n})=0.
    \end{align*}
    Hence, combining \cref{lem:GeneralLowerBound,eq:all_selection_product,eq:cost_and_minimum_cost}, we have
    \begin{align*}
        r_n
        &\ge
        \frac14\closed{\gP(T;\veta_{1:n-1})\,\frac{M_n}{2\open{2+s_{k+1}(T;\veta_{1:n-1})}}}^2
        =
        \frac{M_n^2}{16\,\Psi_2(T;\veta_{1:n-1})^2}.
    \end{align*}
    Choosing $T$ attaining $\gV_2(\veta_{1:n-1})=\min_{T\subseteq[n-1]}\Psi_2(T;\veta_{1:n-1})$ maximizes the right-hand side and yields
    \begin{align*}
        r_n\ge\frac{M_n^2}{16\,\gV_2(\veta_{1:n-1})^2},
        \qquad\implies\qquad
        \gV_2(\veta_{1:n-1})\ge\frac{M_n}{4\sqrt{r_n}}.
    \end{align*}
    Since $r_n \le 1/4$ gives $4\sqrt{r_n}\le 2$, we get $\gV_2(\vxi)\ge M_n/2$, so the hypothesis of \cref{lem:OrderSensitiveSumMaximum} holds with $M=M_n$ and length $m=n-1$. It follows that
    \begin{align*}
        S_n = S_{n-1}+M_n \le C\,\open{(n-1)+1}M_n^{1-\nu_\star}=C\,nM_n^{1-\nu_\star}.
    \end{align*}
    This completes the proof.
\end{proof}

%% file: sec/905TopK.tex
\section{Lower Bounds via Top-k Checkpoint Selection}

Some readers might wonder how we should select the checkpoint steps.
We do not have a definitive answer.
However, this section may provide some guidance.

Before obtaining our main results, we first tried to restrict the strategy for selecting the checkpoint steps.
This is because, previously, we did not know how to handle the combinatorial optimization of the general lower bound (i.e., the right-hand side of the inequality in \cref{lem:GeneralLowerBound}) over all $T\subseteq[N]$.
In particular, we have considered sorting the step-sizes in descending order and selecting the top-$k$ largest steps as checkpoints.
As a result, we obtained slightly worse $\Omega(N^{-1.463})$ non-anytime rate bound (\cref{subsec:topk_nonanytime}) and $\Omega(n^{-1.188})$ anytime barrier (\cref{subsec:topk_anytime}).
We suspect that simple top-$k$ selection loses step-order information, which is important in modern step-size schedules for accelerating GD (e.g., the silver step-size schedule).

\subsection{Non-Anytime Lower Bound via Top-k Checkpoint Selection}\label{subsec:topk_nonanytime}

Let
\[
    \zeta_1 \ge \zeta_2 \ge \cdots \ge \zeta_N > 0
\]
be the decreasing rearrangement of $\eta_1,\ldots,\eta_N$. For any positive sequence $\xi=(\xi_1,\ldots,\xi_N)$ and $k=0,\ldots,N$, define the remaining sum
\begin{align}
    \Sigma_k(\xi) := \sum_{j=k+1}^{N} \xi_j,
    \label{eq:topk_remaining_sum}
\end{align}
and
\begin{align}
    \Phi_k(\xi):=\frac{\left(\prod_{j=1}^{k}\xi_j\right)^2}{4\cdot 16^k\left(1+\dfrac{\Sigma_k(\xi)}{2k+1}\right)^{2k+1}}.
    \label{eq:topk_phi}
\end{align}
\begin{restatable}[Top-$k$ reduction]{lemma}{lemTopKReduction}
\label{lem:TopKReduction}
    For every $k\in\{0,\ldots,N\}$,
    \begin{align}
        \mathcal R_N(\veta) \ge \Phi_k(\zeta).
        \label{eq:topk_reduction}
    \end{align}
    Consequently,
    \begin{align}
        \mathcal R_N(\veta)\ge\max_{0\le k\le N}\Phi_k(\zeta).
        \label{eq:topk_max_phi}
    \end{align}
\end{restatable}
\begin{proof}[Proof of \cref{lem:TopKReduction}]
    The case $k=0$ follows directly from \cref{lem:GeneralLowerBound} by selecting no big step. Hence, fix $k\ge1$ and choose any $k$ indices whose step sizes are $\zeta_1,\ldots,\zeta_k$, ordered according to their original positions in the schedule. For the corresponding block decomposition, the selected big steps $b_1,\ldots,b_k$ and the unselected block sums $s_1,\ldots,s_{k+1}$ satisfy
    \begin{align}
        \prod_{i=1}^{k} b_i=\prod_{j=1}^{k}\zeta_j,\qquad\sum_{i=1}^{k+1}s_i=\Sigma_k(\zeta).
        \label{eq:topk_block_relations}
    \end{align}
    We apply AM--GM to the following $2k+1$ nonnegative numbers:
    \[
        1+\frac{s_1}{2},\ 1+\frac{s_1}{2},\ \ldots,\ 1+\frac{s_k}{2},\ 1+\frac{s_k}{2},\ 1+s_{k+1}.
    \]
    Their average is $1+\Sigma_k(\zeta)/(2k+1)$, and therefore
    \begin{align}
        (1+s_{k+1}) \prod_{i=1}^{k}(s_i+2)^2 &=4^k(1+s_{k+1})\prod_{i=1}^{k}\left(1+\frac{s_i}{2}\right)^2 \nonumber\\
        &\le 4^k\left(1+\frac{\Sigma_k(\zeta)}{2k+1}\right)^{2k+1}.
        \label{eq:topk_amgm}
    \end{align}
    Substituting \cref{eq:topk_block_relations,eq:topk_amgm} into \cref{eq:general_normalized_lower_bound} gives \cref{eq:topk_reduction}. Maximizing over $k$ proves \cref{eq:topk_max_phi}.
\end{proof}
It remains to lower-bound the maximum in \cref{eq:topk_max_phi} uniformly over all decreasing sequences $\zeta$; the following comparison lemma allows us to replace $\zeta$ by a suitably chosen reference sequence.
\begin{restatable}[Product-sum crossing lemma]{lemma}{lemProductSumCrossing}
\label{lem:ProductSumCrossing}
    Let $\zeta_1\ge\cdots\ge\zeta_N>0$ and let $\gamma_1,\ldots,\gamma_N>0$ be arbitrary. Then there exists $k\in\{0,\ldots,N\}$ such that
    \begin{align}
        \prod_{j=1}^{k}\zeta_j\ge\prod_{j=1}^{k}\gamma_j,\qquad\sum_{j=k+1}^{N}\zeta_j\le\sum_{j=k+1}^{N}\gamma_j.
        \label{eq:product_sum_crossing_two_sided}
    \end{align}
    If $\sum_{j=1}^{N}\zeta_j>\sum_{j=1}^{N}\gamma_j$, then $k$ can moreover be chosen so that $k\ge1$ and the prefix-product inequality is strict. Consequently,
    \begin{align}
        \max_{0\le k\le N}\Phi_k(\zeta)
        \ge
        \min_{0\le k\le N}\Phi_k(\gamma).
        \label{eq:product_sum_crossing_phi_comparison}
    \end{align}
\end{restatable}

\begin{proof}[Proof of \cref{lem:ProductSumCrossing}]
    For $k=0,\ldots,N$, use the remaining-sum notation $\Sigma_k(\zeta)$ and $\Sigma_k(\gamma)$ from \cref{eq:topk_remaining_sum}.
    If $\Sigma_0(\zeta)\le\Sigma_0(\gamma)$, then $k=0$ already satisfies \cref{eq:product_sum_crossing_two_sided}. Otherwise, let $k\ge1$ be the smallest index for which $\Sigma_k(\zeta)\le\Sigma_k(\gamma)$. Such an index exists because $\Sigma_N(\zeta)=\Sigma_N(\gamma)=0$.

    For each $i=1,\ldots,k$, define
    \begin{align*}
        \Delta_i:=\sum_{j=i}^{k}(\zeta_j-\gamma_j)=\left(\Sigma_{i-1}(\zeta)-\Sigma_{i-1}(\gamma)\right)-\left(\Sigma_k(\zeta)-\Sigma_k(\gamma)\right).
    \end{align*}
    By the minimality of $k$, all $\Delta_i$ are nonnegative, and they are all strictly positive when $\Sigma_0(\zeta)>\Sigma_0(\gamma)$. Also, $\psi_j:=1/\zeta_j$ is nondecreasing because $\zeta_j$ is nonincreasing. Using the concavity of $\log$ and summation by parts, we obtain
    \begin{align*}
        \sum_{j=1}^{k}\left(\log\gamma_j-\log\zeta_j\right)&\le-\sum_{j=1}^{k}\psi_j(\zeta_j-\gamma_j) \nonumber\\ &=-\psi_1\Delta_1-\sum_{i=2}^{k}(\psi_i-\psi_{i-1})\Delta_i\le 0.
    \end{align*}
    Hence the prefix product inequality in \cref{eq:product_sum_crossing_two_sided} holds, and it is strict when $\Sigma_0(\zeta)>\Sigma_0(\gamma)$; the remaining-sum inequality holds by the choice of $k$.

    Since $\Phi_k$ is increasing in the prefix product and decreasing in the remaining sum, $\Phi_k(\zeta)\ge\Phi_k(\gamma)$, which immediately implies \cref{eq:product_sum_crossing_phi_comparison}.
\end{proof}
We now choose the reference sequence $\gamma$ so that $\Phi_k(\gamma)$ can be bounded uniformly over all $k$.
\begin{restatable}[Polynomially decaying reference sequence]{lemma}{lemPolynomiallyDecayingReferenceSequence}
\label{lem:PolynomiallyDecayingReferenceSequence}
    Let $p>1$ and $\bar\gamma>0$ satisfy
    \begin{align}
        4e^{-p}\le \bar\gamma\le 2(p-1).
        \label{eq:power_law_feasible_interval}
    \end{align}
    For a fixed horizon $N$, define
    \begin{align*}
        \gamma_j:=\bar\gamma\left(\frac{N}{j}\right)^p, \qquad j=1,\ldots,N.
    \end{align*}
    Then
    \begin{align}
        \min_{0\le k\le N}\Phi_k(\gamma)\ge\frac{e^{-2p}}{4}N^{-p}.
        \label{eq:power_law_uniform_phi_bound}
    \end{align}
\end{restatable}
\begin{proof}[Proof of \cref{lem:PolynomiallyDecayingReferenceSequence}]
    We first consider $k=0$. Since $p>1$,
    \begin{align*}
        \sum_{j=1}^{N}j^{-p} \le 1+\int_{1}^{\infty}x^{-p}\,dx=1+\frac{1}{p-1}.
    \end{align*}
    Using $\bar\gamma\le2(p-1)$, we therefore have
    $\sum_{j=1}^{N}\gamma_j\le2pN^p$. Consequently,
    \begin{align}
        \Phi_0(\gamma)
        &=
        \frac{1}{4\left(1+\sum_{j=1}^{N}\gamma_j\right)}\ge\frac{1}{4(1+2pN^p)} \nonumber\\
        &\ge\frac{1}{4(2p+1)}N^{-p}\ge\frac{e^{-2p}}{4}N^{-p},
        \label{eq:power_law_k_zero}
    \end{align}
    where the last inequality follows from $e^{2p}\ge1+2p$.
    
    Now fix $1\le k\le N$. The sum of the remaining terms of the comparison sequence satisfies
    \begin{align}
        \sum_{j=k+1}^{N}\gamma_j
        &=
        \bar\gamma N^p\sum_{j=k+1}^{N}j^{-p}\le \bar\gamma N^p\int_k^N x^{-p}\,dx \nonumber\\
        &=
        \frac{\bar\gamma N^p}{p-1}\left(k^{1-p}-N^{1-p}\right)\le 2k\left(\frac Nk\right)^p-2N.
        \label{eq:power_law_remaining_sum_bound}
    \end{align}
    If $Y:=(N/k)^p\ge1$, then
    \[
        1+\frac{2kY-2N}{2k+1}\le Y.
    \]
    Combining this observation with \cref{eq:power_law_remaining_sum_bound} gives
    \begin{align}
        1+
        \frac{1}{2k+1}\sum_{j=k+1}^{N}\gamma_j
        \le
        \left(\frac Nk\right)^p.
        \label{eq:power_law_denominator_bound}
    \end{align}
    On the other hand, the prefix product is
    \[
        \prod_{j=1}^{k}\gamma_j
        =
        \frac{(\bar\gamma N^p)^k}{(k!)^p}.
    \]
    Substitution into \cref{eq:topk_phi}, followed by
    \cref{eq:power_law_denominator_bound}, yields
    \begin{align*}
        \Phi_k(\gamma)
        &\ge \frac{(\bar\gamma N^p)^{2k}}{4\cdot 16^k(k!)^{2p}(N/k)^{p(2k+1)}}\nonumber\\
        &\ge
        \frac{N^{-p}}{4}\,
        k^p\left(\frac{\bar\gamma}4\right)^{2k}
        \left(\frac{k^k}{k!}\right)^{2p}.
    \end{align*}
    We use the standard upper Stirling bound
    \[
        k!\le e\sqrt{k}\left(\frac{k}{e}\right)^k,
    \]
    which implies
    \[
        k^p\left(\frac{k^k}{k!}\right)^{2p}
        \ge e^{2p(k-1)}.
    \]
    Thus
    \begin{align*}
        \Phi_k(\gamma)
        &\ge
        \frac{e^{-2p}}{4}N^{-p}
        \left(\frac{\bar\gamma e^p}{4}\right)^{2k}
        \ge
        \frac{e^{-2p}}{4}N^{-p},
    \end{align*}
    where the last inequality uses $\bar\gamma\ge4e^{-p}$. Together with
    \cref{eq:power_law_k_zero}, this proves
    \cref{eq:power_law_uniform_phi_bound}.
\end{proof}
The feasible interval in \cref{eq:power_law_feasible_interval} is nonempty precisely when $4e^{-p}\le2(p-1)$. The smallest exponent allowed by this comparison argument is therefore obtained when the two endpoints coincide.
\begin{restatable}[Finite-horizon lower bound via top-$k$ selection]{theorem}{thmTopKFiniteHorizon}
\label{thm:TopKFiniteHorizon}
    Define
    \begin{align}
        \alpha:=1+W_0\left(\frac{2}{e}\right)=1.463055513\ldots, \qquad
        \nu:=\frac{1}{\alpha}=0.683501064\ldots.
        \label{eq:topk_critical_exponents}
    \end{align}
    where $W_0$ denotes the principal branch of the Lambert $W$ function. Then every positive $N$-step schedule $\veta\in(0,\infty)^N$ satisfies
    \begin{align}
        \mathcal R_N(\veta)
        \ge
        \frac{e^{-2\alpha}}{4}
        N^{-\alpha}.
        \label{eq:topk_finite_horizon_lower_bound}
    \end{align}
    Consequently, no horizon-dependent family of positive step-size schedules can achieve a worst-case rate of $o\open{N^{-\alpha}}$.
\end{restatable}
\begin{proof}[Proof of \cref{thm:TopKFiniteHorizon}]
    The definition of $\alpha$ is equivalent to
    \begin{align}
        (\alpha-1)e^{\alpha}=2.
        \label{eq:topk_critical_balance}
    \end{align}
    Hence, if
    \[
        \bar\gamma:=2(\alpha-1)=4e^{-\alpha},
    \]
    then $\bar\gamma$ satisfies \cref{eq:power_law_feasible_interval} with equality at both endpoints. We may therefore apply \cref{lem:ProductSumCrossing} to the decreasing rearrangement $\zeta$ of the given schedule and the reference sequence
    \[
        \gamma_j=\bar\gamma
        \left(\frac Nj\right)^{\alpha}.
    \]
    Combining \cref{lem:TopKReduction,lem:ProductSumCrossing,lem:PolynomiallyDecayingReferenceSequence} gives
    \begin{align*}
        \mathcal R_N(\veta)
        \ge
        \max_{0\le k\le N}\Phi_k(\zeta)
        \ge
        \min_{0\le k\le N}\Phi_k(\gamma)
        \ge
        \frac{e^{-2\alpha}}{4}
        N^{-\alpha},
    \end{align*}
    proving \cref{eq:topk_finite_horizon_lower_bound}. Since the bound is uniform over every positive schedule of length $N$, the final statement follows immediately.
\end{proof}

\subsection{Anytime Lower Bound via Top-k Checkpoint Selection}
\label{subsec:topk_anytime}

For an infinite step-size schedule, we retain the notation $r_n$, $S_n$, $M_n$, and the notion of a record time introduced in \cref{sec:anytime}.

\begin{restatable}[Top-$k$ constraints at record times]{lemma}{lemTopKRecordConstraints}
\label{lem:TopKRecordConstraints}
    Fix a record time $n$, and let
    \[
    \zeta_{n,1}\ge\zeta_{n,2}\ge \cdots \ge \zeta_{n,n}>0
    \]
    be the decreasing rearrangement of $\eta_1, \ldots, \eta_n$. For $k=1,\ldots, n$, define
    \begin{align*}
        \Sigma_{n,k}:=\sum_{j=k+1}^n \zeta_{n,j},\qquad G_{n,k}:=\left(\prod_{j=1}^k \zeta_{n,j}\right)^{1/k}.
    \end{align*}
    Then, for every $k=1,\ldots, n$,
    \begin{align}
        r_n\ge \frac 14 \left(\frac{G_{n,k}}{4+\dfrac{2\Sigma_{n,k}}{k}}\right)^{2k}
        \label{eq:anytime_topk_record_lower_bound}
    \end{align}
    Hence, if $r_n<1/4$, then
    \begin{align}
        G_{n,k}\le 4+\frac{2\Sigma_{n,k}}{k},\qquad k=1,\ldots,n.
        \label{eq:anytime_topk_record_constraint}
    \end{align}
    Moreover, every record time satisfies
    \begin{align}
        r_n\ge\frac{M_n^2}{16(2+S_{n-1})^2}.
        \label{eq:anytime_single_record_step}
    \end{align}
\end{restatable}
\begin{proof}[Proof of \cref{lem:TopKRecordConstraints}]
    Fix $k\in[n]$. Since $n$ is a record time, we may choose the $k$ largest step-sizes among the first $n$ steps as the selected big steps in such a way that the last step $\eta_n=M_n$ is selected. Thus the terminal small-step block is empty, while the total sum of the unselected steps is $\Sigma_{n,k}$. If $s_1,\ldots,s_k$ are the small-step block sums preceding the selected steps, then
    \begin{align}
        \prod_{i=1}^{k}b_i=\prod_{j=1}^{k}\zeta_{n,j}=G_{n,k}^{\,k},\qquad \sum_{i=1}^{k}s_i=\Sigma_{n,k},\qquad s_{k+1}=0.
        \label{eq:anytime_record_block_relations}
    \end{align}
    Applying AM--GM to $2+s_1,\ldots,2+s_k$ gives
    \begin{align}
        \prod_{i=1}^{k}(2+s_i)\le\left(2+\frac{\Sigma_{n,k}}{k}\right)^k.
        \label{eq:anytime_record_amgm}
    \end{align}
    Substituting \cref{eq:anytime_record_block_relations,eq:anytime_record_amgm} into \cref{eq:general_normalized_lower_bound} yields \cref{eq:anytime_topk_record_lower_bound}. If $r_n<1/4$, then $(4r_n)^{1/(2k)}<1$, so \cref{eq:anytime_topk_record_lower_bound} immediately gives \cref{eq:anytime_topk_record_constraint}.

    Finally, selecting only the last step $\eta_n=M_n$ gives a single big step with preceding small-step sum $S_{n-1}$ and an empty terminal block. Applying \cref{lem:GeneralLowerBound} once more gives \cref{eq:anytime_single_record_step}.
\end{proof}
The constraints in \cref{eq:anytime_topk_record_constraint} couple the prefix geometric means with the corresponding remaining sums. To analyze them simultaneously, we construct a comparison sequence that attains equality in these constraints at every index. This sequence will serve as an extremal reference in the sum--maximum argument below, and its growth is controlled by the same exponent $\alpha$ introduced in \cref{eq:topk_critical_exponents}.

\begin{restatable}[Critical comparison sequence]{lemma}{lemCriticalComparisonSequence}
\label{lem:CriticalComparisonSequence}
    For every $n\ge1$, there exist positive numbers $\kappa_1,\ldots,\kappa_n$ and numbers $P_1>P_2>\cdots>P_n=4$ such that, for every $k=1,\ldots,n$,
    \begin{align}
        \left(\prod_{j=1}^{k}\kappa_j\right)^{1/k}=P_k=4+\frac{2}{k}\sum_{j=k+1}^{n}\kappa_j.
        \label{eq:critical_sequence_equalities}
    \end{align}
    Moreover, there exists a numerical constant $C_0>0$ such that
    \begin{align}
        P_k\le C_0\left(\frac{n}{k}\right)^{\alpha},\qquad k=1,\ldots,n.
        \label{eq:critical_sequence_power_bound}
    \end{align}
\end{restatable}
\begin{proof}[Proof of \cref{lem:CriticalComparisonSequence}]
    Set $P_n=4$. For $k=n-1,\ldots,1$, define $P_k>P_{k+1}$ as the unique solution of
    \begin{align}
        2\frac{P_{k+1}^{k+1}}{P_k^k}=kP_k-(k+1)P_{k+1}+4.
        \label{eq:critical_sequence_recursion}
    \end{align}
    To see that this solution exists and is unique, fix $q=P_{k+1}\ge4$ and consider
    \[
        h(P):=kP-(k+1)q+4-2q^{k+1}P^{-k}.
    \]
    The function $h$ is strictly increasing on $(0,\infty)$, satisfies $h(q)=4-3q<0$, and tends to $+\infty$ as $P\to\infty$. Hence there is a unique root larger than $q$.

    Define
    \begin{align*}
        \kappa_1:=P_1,\qquad \kappa_{k+1}:=\frac{P_{k+1}^{k+1}}{P_k^k},\qquad k=1,\ldots,n-1.
    \end{align*}
    The products telescope, and a backward induction using \cref{eq:critical_sequence_recursion} gives
    \begin{align*}
        \prod_{j=1}^{k}\kappa_j=P_k^k,\qquad \sum_{j=k+1}^{n}\kappa_j=\frac{k}{2}(P_k-4),\qquad k=1,\ldots,n.
    \end{align*}
    These identities prove \cref{eq:critical_sequence_equalities}.

    We now establish the growth bound. For $k=1,\ldots,n-1$, set
    \begin{align*}
        \rho_k:=\frac{P_{k+1}}{P_k},\qquad D_k:=(k+1)(1-\rho_k).
    \end{align*}
    Dividing \cref{eq:critical_sequence_recursion} by $P_k$ yields
    \begin{align*}
        2\rho_k^{k+1}=D_k-1+\frac{4}{P_k}.
    \end{align*}
    Since $P_k\ge4$ and $\rho_k=1-D_k/(k+1)$, we obtain
    \begin{align*}
        D_k-1<2\rho_k^{k+1}\le2e^{-D_k}.
    \end{align*}
    The function $D\mapsto D-1-2e^{-D}$ is strictly increasing and vanishes at $\alpha$, because \cref{eq:topk_critical_balance} is equivalent to $\alpha-1=2e^{-\alpha}$. Hence $D_k<\alpha$, and therefore
    \begin{align*}
        \rho_k>1-\frac{\alpha}{k+1}.
    \end{align*}
    Using $P_n=4$ and telescoping the ratios, we have
    \begin{align}
        P_k=4\prod_{j=k+1}^{n}\rho_{j-1}^{-1}\le4\prod_{j=k+1}^{n}\left(1-\frac{\alpha}{j}\right)^{-1}.
        \label{eq:critical_sequence_product_bound}
    \end{align}
    Since $\alpha<2$, the arguments $\alpha/j$ are uniformly bounded away from $1$ for $j\ge2$. Thus there is a numerical constant $\widetilde C_1>0$ such that $-\log(1-x)\le x+\widetilde C_1x^2$ throughout the required range. Applying this to \cref{eq:critical_sequence_product_bound} and using the standard harmonic-sum estimates gives
    \begin{align*}
        \log\frac{P_k}{4}&\le \alpha\sum_{j=k+1}^{n}\frac1j+\widetilde C_1\alpha^2\sum_{j=k+1}^{n}\frac1{j^2}\le \alpha\log\frac{n}{k}+\widetilde C_2,
    \end{align*}
    for a numerical constant $\widetilde C_2>0$. Exponentiating proves \cref{eq:critical_sequence_power_bound}.
\end{proof}
\begin{restatable}[Sum--maximum bound]{lemma}{lemTopKSumMaximumBound}
\label{lem:TopKSumMaximumBound}
    Let $\zeta_1\ge\cdots\ge\zeta_n>0$, and define
    \[
        M:=\zeta_1,\qquad S:=\sum_{j=1}^{n}\zeta_j.
    \]
    Suppose that, for every $k=1,\ldots,n$,
    \begin{align}
        \left(\prod_{j=1}^{k}\zeta_j\right)^{1/k}\le4+\frac{2}{k}\sum_{j=k+1}^{n}\zeta_j.
        \label{eq:sum_max_assumption}
    \end{align}
    Then there exists a numerical constant $C>0$ such that
    \begin{align}
        S\le C\left(nM^{1-\nu}+M+n\right).
        \label{eq:topk_sum_max_bound}
    \end{align}
\end{restatable}
\begin{proof}[Proof of \cref{lem:TopKSumMaximumBound}]
    Let $(\kappa_j)_{j=1}^{n}$ and $(P_k)_{k=1}^{n}$ be the sequences from \cref{lem:CriticalComparisonSequence}. If $M\le4$, then $S\le4n$, so \cref{eq:topk_sum_max_bound} is immediate. Assume from now on that $M>4$. If $M\ge P_1$, set $k_0=1$; otherwise, choose the unique $k_0\in\{2,\ldots,n\}$ such that
    \[
        P_{k_0}\le M<P_{k_0-1}.
    \]
    Define a comparison sequence $\widehat\kappa_1,\ldots,\widehat\kappa_n$ by
    \begin{align*}
        \widehat\kappa_j:=\begin{cases}
            M, & j\le k_0,\\
            \kappa_j, & j>k_0.
        \end{cases}
    \end{align*}
    For every $k>k_0$, \cref{eq:critical_sequence_equalities} and $M\ge P_{k_0}$ imply
    \begin{align}
        \left(\prod_{j=1}^{k}\widehat\kappa_j\right)^{1/k}=P_k\left(\frac{M}{P_{k_0}}\right)^{k_0/k}\ge P_k=4+\frac{2}{k}\sum_{j=k+1}^{n}\widehat\kappa_j.
        \label{eq:sum_max_gamma_reverse_constraint}
    \end{align}
    Moreover,
    \begin{align}
        \sum_{j=1}^{n}\widehat\kappa_j=k_0M+\frac{k_0}{2}(P_{k_0}-4)\le\frac32k_0M.
        \label{eq:sum_max_gamma_sum}
    \end{align}

    We claim that $S\le\sum_{j=1}^{n}\widehat\kappa_j$. Suppose otherwise. Since $\zeta$ is decreasing and $S>\sum_j\widehat\kappa_j$, the strict form of \cref{lem:ProductSumCrossing} gives an index $k\in[n]$ such that
    \begin{align}
        \prod_{j=1}^{k}\zeta_j>\prod_{j=1}^{k}\widehat\kappa_j,\qquad \sum_{j=k+1}^{n}\zeta_j\le\sum_{j=k+1}^{n}\widehat\kappa_j.
        \label{eq:sum_max_product_sum_crossing}
    \end{align}
    The inequality $k\le k_0$ is impossible because $\zeta_j\le M=\widehat\kappa_j$ for every $j\le k_0$, so necessarily $k>k_0$. Using \cref{eq:sum_max_assumption,eq:sum_max_product_sum_crossing,eq:sum_max_gamma_reverse_constraint}, we then obtain
    \begin{align*}
        \left(\prod_{j=1}^{k}\zeta_j\right)^{1/k}&\le4+\frac{2}{k}\sum_{j=k+1}^{n}\zeta_j\le4+\frac{2}{k}\sum_{j=k+1}^{n}\widehat\kappa_j\le\left(\prod_{j=1}^{k}\widehat\kappa_j\right)^{1/k},
    \end{align*}
    contradicting the strict prefix-product inequality in \cref{eq:sum_max_product_sum_crossing}. Therefore $S\le\sum_j\widehat\kappa_j$.

    It remains to bound $k_0$. If $k_0\ge2$, then $M<P_{k_0-1}$ and \cref{eq:critical_sequence_power_bound} give
    \begin{align}
        M<C_0\left(\frac{n}{k_0-1}\right)^{\alpha},\qquad\text{hence}\qquad k_0-1\le C_0^{\nu}nM^{-\nu}.
        \label{eq:sum_max_m_bound}
    \end{align}
    Combining \cref{eq:sum_max_gamma_sum,eq:sum_max_m_bound} yields $S\le C'(M+nM^{1-\nu})$ for a numerical constant $C'$. The case $k_0=1$ satisfies the same bound directly, and adding the previously handled case $M\le4$ proves \cref{eq:topk_sum_max_bound}.
\end{proof}
We can now combine the preceding estimates. The exponent obtained from the top-$k$ analysis is
\begin{align*}
    \beta:=\frac{2\alpha}{1+\alpha}=\frac{2}{1+\nu}=1.188000437\ldots,
\end{align*}
which, since $\nu=1/\alpha$, is equivalent to
\begin{align}
    \frac{1}{\beta}=\frac{1+\alpha}{2\alpha}=\frac{1+\nu}{2}.
    \label{eq:topk_anytime_exponent_identity}
\end{align}
\begin{restatable}[Anytime lower bound via top-$k$ selection]{theorem}{thmTopKAnytime}
\label{thm:TopKAnytime}
    For every fixed positive infinite step-size schedule $\veta=(\eta_t)_{t\ge1}$,
    \begin{align}
        \limsup_{n\to\infty}n^{\beta}\mathcal R_n(\veta_{1:n})>0.
        \label{eq:topk_anytime_limsup}
    \end{align}
    Equivalently, no positive infinite step-size schedule satisfies
    \begin{align}
        \mathcal R_n(\veta_{1:n})=o\open{n^{-\beta}}.
        \label{eq:topk_anytime_no_little_o}
    \end{align}
\end{restatable}

\begin{proof}[Proof of \cref{thm:TopKAnytime}]
    Suppose for contradiction that $r_n=o(n^{-\beta})$. Selecting no big step in \cref{lem:GeneralLowerBound} gives $r_n\ge[4(1+S_n)]^{-1}$. Since $r_n\to0$, for all sufficiently large $n$ this implies
    \begin{align}
        S_n\ge\frac{1}{4r_n}-1\ge\frac{1}{8r_n}.
        \label{eq:topk_anytime_sum_lower}
    \end{align}
    Therefore $S_n=\omega(n^{\beta})$, and hence
    \begin{align*}
        M_n\ge\frac{S_n}{n}=\omega\open{n^{\beta-1}}\longrightarrow\infty.
    \end{align*}
    In particular, there are infinitely many record times.

    Fix a sufficiently large record time $n$. Since $r_n<1/4$, \cref{lem:TopKRecordConstraints} gives \cref{eq:anytime_topk_record_constraint}, so \cref{lem:TopKSumMaximumBound} applied to the decreasing rearrangement of the first $n$ steps yields
    \begin{align}
        S_n\le C\left(nM_n^{1-\nu}+M_n+n\right).
        \label{eq:topk_anytime_sum_upper}
    \end{align}
    At the same time, \cref{eq:anytime_single_record_step} implies
    \[
        M_n\le4\,(2+S_{n-1})\sqrt{r_n}.
    \]
    Since $S_n\to\infty$ along the record times under consideration and $S_{n-1}\le S_n$, there exists a numerical constant $C_1>0$ such that, for all sufficiently large record times,
    \begin{align}
        M_n\le C_1S_n\sqrt{r_n}.
        \label{eq:topk_anytime_max_upper}
    \end{align}
    Substituting \cref{eq:topk_anytime_max_upper} into \cref{eq:topk_anytime_sum_upper}, dividing by $S_n$, and then using \cref{eq:topk_anytime_sum_lower} gives
    \begin{align}
        1&\le C_2\left(nS_n^{-\nu}r_n^{(1-\nu)/2}+\sqrt{r_n}+\frac{n}{S_n}\right)\nonumber\\
        &\le C_3\left(nr_n^{\nu+(1-\nu)/2}+\sqrt{r_n}+nr_n\right)\nonumber\\
        &=C_3\left(nr_n^{1/\beta}+\sqrt{r_n}+nr_n\right),
        \label{eq:topk_anytime_master_inequality}
    \end{align}
    for numerical constants $C_2,C_3>0$, where the equality uses $\nu+\frac{1-\nu}{2}=\frac{1+\nu}{2}=\frac{1}{\beta}$ from \cref{eq:topk_anytime_exponent_identity}.
    Under the assumed rate $r_n=o(n^{-\beta})$, all three terms on the right-hand side of \cref{eq:topk_anytime_master_inequality} vanish along the infinite record-time subsequence:
    \begin{align*}
        nr_n^{1/\beta}=(n^{\beta}r_n)^{1/\beta}\longrightarrow0,\qquad \sqrt{r_n}\longrightarrow0,\qquad nr_n=n^{1-\beta}(n^{\beta}r_n)\longrightarrow0.
    \end{align*}
    This contradicts \cref{eq:topk_anytime_master_inequality} and proves \cref{eq:topk_anytime_no_little_o}, equivalently \cref{eq:topk_anytime_limsup}.
\end{proof}